\documentclass[12pt]{amsart}
\usepackage{amsmath,amssymb}
\usepackage{amsfonts}
\usepackage{amsthm}
\usepackage{latexsym}
\usepackage{graphicx}

\def\p{\partial}
\def\tr{\mbox{tr}}

\def\R{\mathbb{R}}

\def\vv<#1>{\langle#1\rangle}
\def\ol{\overline}

\def\bN{\mathbf{N}}
\def\bd{\mathbf{D}}
\def\1{\mathbf{1}}
\def\bA{\mathbf{A}}

\def\XXint#1#2{\setbox0=\hbox{$#1{#2}{\int}$}{#2}\kern-.5\wd0 }

\def\XXint#1#2#3{{\setbox0=\hbox{$#1{#2#3}{\int}$}
     \vcenter{\hbox{$#2#3$}}\kern-.5\wd0}}

\def\vv<#1>{{\left\langle#1\right\rangle}}

\def\sgn{{\rm sgn}}
\newtheorem{thm}{Theorem}[section]

\newtheorem{lem}{Lemma}[section]
\newtheorem{prop}{Proposition}[section]
\newtheorem{cor}{Corollary}[section]
\theoremstyle{definition}
\newtheorem{defn}{Definition}[section]
\theoremstyle{remark}

\numberwithin{equation}{section}

\begin{document}
\title{Higher order Dirichlet-to-Neumann maps on graphs and their eigenvalues}

\author{Yongjie Shi$^1$}
\address{Department of Mathematics, Shantou University, Shantou, Guangdong, 515063, China}
\email{yjshi@stu.edu.cn}
\author{Chengjie Yu$^2$}
\address{Department of Mathematics, Shantou University, Shantou, Guangdong, 515063, China}
\email{cjyu@stu.edu.cn}
\thanks{$^1$Research partially supported by NSF of China with contract no. 11701355. }
\thanks{$^2$Research partially supported by NSF of China with contract no. 11571215.}
\renewcommand{\subjclassname}{%
  \textup{2010} Mathematics Subject Classification}
\date{}
\begin{abstract}
In this paper, we first introduce higher order Dirichlet-to-Neumann maps on graphs which can be viewed as a discrete analogue of the corresponding Dirichlet-to-Neumann maps on compact Riemannian manifolds with boundary and a higher order generalization of the Dirichlet-to-Neumann map on graphs introduced by Hua-Huang-Wang\cite{HHW} and Hassannezhad-Miclo \cite{HM}. Then, some Raulot-Savo-type estimates on the eigenvalues of the DtN maps introduced are derived.
\end{abstract}
\maketitle\markboth{Shi \& Yu}{Higher order DtN maps}
\section{Introduction}
Let $(M^n,g)$ be a compact Riemannian manifold with boundary and $u$ be a harmonic function on $M$. Then the classical Steklov operator maps $u|_{\p M}$ to $\frac{\p u}{\p n}$. The eigenvalues of this operator are called the Steklov eigenvalues of $(M^n,g)$. These kinds of notions were first introduced by Steklov \cite{Ku,St} when studying the frequency of liquid sloshing. They were also found to have deep relationship with the Calder\'on inverse problem in applied mathematics (\cite{Ca,Ul}). There has been many works on the Steklov eigenvalues. For more details of the topic, one may refer to the survey article \cite{GP}. Steklov operators are also called Dirichlet-to-Neumann maps and usually abbreviated as DtN maps because they transform Dirchlet boundary data of a harmonic function to its Neumann boundary data.

Recently, a discrete version of Steklov's operator on graphs was introduced by Hua-Huang-Wang \cite{HHW} and Hassannezhad-Miclo \cite{HM} independently. In \cite{HHW}, the authors also obtained some Cheeger-type inequalities for the first eigenvalue of DtN maps on graphs that are discrete analogue of previous results of Escobar \cite{Es} and Jammes \cite{Ja}. In \cite{HM}, the authors also obtained some Cheeger-type inequalities for higher order eigenvalues. In \cite{Pe}, the author obtained some interesting lower bounds for Steklov eigenvalues of graphs. In \cite{HHW2}, the authors introduced the Steklov eigenvalues for infinite subgraphs and obtained some Cheeger-type inequalities for infinite subgraphs. Recently, in \cite{HH}, the authors derived a discrete version of  Brock's estimate (see \cite{Br}).

Note that the classical Steklov operators and eigenvalues for functions were extended to differential forms by Raulot-Savo \cite{RS2}. In fact,  different definition of DtN maps for differential forms were introduced in \cite{BS} and \cite{JL} before Raulot-Savo \cite{RS2}. However, the definitions of DtN maps in \cite{BS} and \cite{JL} were motivated by inverse problems in applied mathematics which are not suitable for spectral analysis. Recently, Karpukhin \cite{Ka} slightly modified the DtN maps in \cite{BS} so that it is suitable for spectral analysis and extended the higher dimensional generalization of Hersch-Payne-Schiffer inequality (see \cite{HPS}) in \cite{YY2} to more general cases. According to all of these works, it is a natural problem to develop a discrete version of higher order DtN maps for graphs. This is the motivation of the paper.

It seems that a discrete version of Hodge theory on simplicial complexes was  known to experts and goes back to the pioneer work \cite{Do} of Dodziuk. For its analogue on graphs, although looks very similar with the  Hodge theory on simplicial complexes, the authors find it only appeared in some recent works in applied mathematics. For example, in \cite{JLYY} (see also \cite{Li}), the authors developed a very interesting ranking method which is called the {\bf HodgeRank} by using the Hodge decomposition on graphs.

In this paper, motivated by the works of Hua-Huang-Wang \cite{HHW} and Raulot-Savo \cite{RS2}, we first introduce  higher order DtN maps on graphs.

For a simple graph $G$, we denote by $V(G)$  and $E(G)$ the set of vertices and the set of edges for $G$ respectively. We also adopt the notion of graph with boundary in \cite{Pe}. A pair $(G,B)$ is said to be a graph with boundary $B$ if $G$ is a simple graph and $B\subset V(G)$, such that any two vertices of $B$ are not adjacent to each other, and each vertex in $B$ is adjacent to some vertcies in $\Omega:=B^c$ which is called the interior of $(G,B)$.

A $k$-form $f$ of a graph $G$ is a skew symmetric function with $k+1$ variables on $V(G)$ such that $f(v_0,v_1,\cdots,v_k)=0$ if $\{v_0,v_1,\cdots,v_k\}$ is not a clique in $G$. The collection of all $k$-forms on $G$ is denoted as $A^k(G)$. The exterior differential $df$ of a $k$-form $f$ is a $(k+1)$-form with
\begin{equation}
df(v_0,v_1,\cdots,v_{k+1})=\sum_{i=0}^{k+1}(-1)^if(v_0,\cdots,\hat v_i,\cdots,v_{k+1})
\end{equation}
for any $(k+2)$-clique $\{v_0,v_1,\cdots,v_{k+1}\}$ of $G$.

For a weight $w$ on a graph $G$, we mean a positive function defined on the collection of cliques for $G$. Define the inner product on $A^k(G)$ as
\begin{equation}
\begin{split}
&\vv<\alpha,\beta>\\
=&\frac{1}{(k+1)!}\sum_{v_0,v_1,\cdots,v_k\in G}\alpha(v_0,v_1,\cdots,v_k)\beta(v_0,v_1,\cdots,v_k)w(v_0,v_1,\cdots,v_k).
\end{split}
\end{equation}
Then, for a finite graph $G$ with weight $w$, define $\delta$ to be the adjoint operator of $d$ according to the inner product just defined. The Hodge Laplacian operator is defined to be $\Delta=\delta d+d\delta$ as usual.

For a finite graph $(G,B)$ with boundary $B$, a clique with one vertex in $B$ is called a boundary clique. The restriction of a $k$-form of $G$ to the collection of boundary cliques is called a boundary $k$-form. The collection of boundary $k$-form is denoted as $A^k(\p\Omega)$. Let $w$ be a weight on $G$. Define the inner product on $A^k(\p\Omega)$ as
\begin{equation}
\begin{split}
&\vv<\alpha,\beta>_{\p\Omega}\\
=&\frac{1}{k!}\sum_{v\in\delta\Omega}\sum_{v_1,v_2,\cdots,v_k\in\Omega}\alpha(v,v_1,v_2,\cdots,v_k)\beta(v,v_1,v_2,\cdots,v_k)w(v,v_1,v_2,\cdots,v_k).
\end{split}
\end{equation}
Then, we have the following two  Green's formulas  for finite weighted graphs with boundary:
\begin{equation}
\vv<\delta\alpha,\beta>_\Omega=\vv<\alpha,d\beta>_{G}-\vv<\mathbf{N}\alpha,\beta>_{\p\Omega}
\end{equation}
and
\begin{equation}
\vv<d\alpha,\beta>_{\Omega}=\vv<\alpha,\delta\beta>_\Omega-\vv<\mathbf{D}\alpha,\beta>_{\p\Omega}.\\
\end{equation}
For the expressions of the operators $\bN$ and $\bd$, see \eqref{eq-def-N} and \eqref{eq-def-D}. Motivated by the Green's formulas above, we introduced three kinds of DtN maps $T_{\delta d}^{(k)}$, $T_{d\delta}^{(k)}$ and $T^{(k)}$ for $k$-forms on graphs.

For the definition of $T_{\delta d}^{(k)}$, for each $\varphi\in A^k(\p\Omega)$, let $E_{\delta d}^{(k)}(\varphi)$ be a solution of the boundary value problem:
\begin{equation}
\left\{\begin{array}{ll}
\delta d\omega=0&\mbox{in $\Omega$}\\
\omega|_{\p\Omega}=\varphi
\end{array}\right.
\end{equation}
and define
\begin{equation}
T_{\delta d}^{(k)}(\varphi)=\bN dE_{\delta d}^{(k)}(\varphi).
 \end{equation}
 The existence of $E_{\delta d}^{(k)}(\varphi)$ is guaranteed by Theorem \ref{thm-existence}. Note that the solution of the boundary value problem may not be unique. So, to make sure that $T_{\delta d}^{(k)}(\varphi)$ is well defined, we must verify that $\bN dE_{\delta d}^{(k)}(\varphi)$ is independent of the choice of solution to the boundary value problem. This is done in Theorem \ref{thm-well-define}.

The definitions for $T_{d\delta}^{(k)}$ and $T^{(k)}$ are similar. Let $E_{d\delta}^{(k)}(\varphi)$ and $E^{(k)}(\varphi)$ be solutions of the boundary value problems:
\begin{equation}
\left\{\begin{array}{ll}
d\delta \omega=0&\mbox{in $\Omega$}\\
\omega|_{\p\Omega}=\varphi
\end{array}\right.
\end{equation}
and
\begin{equation}
\left\{\begin{array}{ll}
\Delta \omega=0&\mbox{in $\Omega$}\\
\omega|_{\p\Omega}=\varphi
\end{array}\right.
\end{equation}
respectively. Then, we define
\begin{equation}T_{d\delta}^{(k)}(\varphi)=\bd \delta E_{d\delta}^{(k)}(\varphi)
\end{equation}and
\begin{equation}T^{(k)}(\varphi)=\bd \delta E^{(k)}(\varphi)+\bN dE^{(k)}(\varphi).\end{equation}

By the Green's formulas above, we have
\begin{equation}
\vv<T_{\delta d}^{(k)}\varphi,\psi>_{\p\Omega}=\vv<dE_{\delta d}^{(k)}(\varphi),dE_{\delta d}^{(k)}(\psi)>_{G},
\end{equation}
\begin{equation}
\vv<T_{d\delta }^{(k)}\varphi,\psi>_{\p\Omega}=\vv<\delta E_{d\delta }^{(k)}(\varphi),\delta E_{d\delta }^{(k)}(\psi)>_{\Omega}
\end{equation}
and
\begin{equation}
\vv<T^{(k)}\varphi,\psi>_{\p\Omega}=\vv<\delta E^{(k)}(\varphi),\delta E^{(k)}(\psi)>_{\Omega}+\vv<d E^{(k)}(\varphi),d E^{(k)}(\psi)>_{G}.
\end{equation}
So, $T_{\delta d}^{(k)}$, $T_{d\delta}^{(k)}$ and $T^{(k)}$ are all nonnegative self-adjoint operators. In fact, the three identities above are the motivations of the definitions of the three kinds of DtN maps. $T^{(k)}$ can be viewed as a discrete analogue of Raulot-Savo's DtN maps for differential forms. It seems that $T_{\delta d}^{(k)}$ corresponds to the modification of Belishev-Sharafutdino's DtN maps in \cite{BS} by Karpukhin in \cite{Ka}. Moreover, it is clear that $T_{\delta d}^{(0)}=T^{(0)}$ is the same as the DtN maps introduced in \cite{HHW} and \cite{HM}.

The second part of this paper is to derive some discrete version of Raulot-Savo-type estimates in \cite{RS1,SY,YY} for subgraphs of integer lattices or standard tessellation of $\R^n$. For a nonnegative self-adjoint operator $T$ on a finite dimensional vector space with inner product, we denote by $\lambda_i(T)$ the $i$-th positive eigenvalue of $T$ (in ascending order and counting multiplicities). Moreover, for a graph $G$ and a subset $\Omega\subset V(G)$, as in \cite{HHW}, we denote by $\tilde\Omega$ the subgraph of $G$ with the set of vertices $\ol\Omega=\Omega\cap\delta\Omega$ and set of edges $E(\Omega,\ol\Omega)$ where $\delta\Omega$ means the set of vertices that are not in $\Omega$ but adjacent to some vertices in $\Omega$ and $E(\Omega,\ol\Omega)$ means the set of edges with one end in $\Omega$ and the other in $\ol\Omega$. It is clear that $(\tilde\Omega,\delta\Omega)$ is a graph with boundary $\delta\Omega$.

Let $\Omega$ be a nonempty finite subset of the integer lattice in $\R^n$ and let $T^{(0)}$ be the DtN maps defined above for $(\tilde\Omega,\delta\Omega)$ with normalized weight. Then, using similar arguments as in \cite{RS1,SY,YY}, we have the following estimate:
\begin{equation}
\lambda_1(T^{(0)})+\lambda_2(T^{(0)})+\cdots+\lambda_n(T^{(0)})\leq \frac{|\delta\Omega|}{\min_{1\leq i\leq n}\{|E_i(\tilde\Omega)|\}}.
\end{equation}
where $E_i(\tilde\Omega)$ means the set of edges in $\tilde\Omega$ that are parallel to $e_i$. Here $\{e_1,e_2,\cdots,e_n\}$ is the standard basis of $\R^n$.

Furthermore, for a nonempty finite subset $\Omega$ of the integer lattices which is also the set of vertices of the graph of standard tessellation of $\R^n$, consider $\tilde\Omega$ as the subgraph of the graph of standard tessellation of $\R^n$ with unit weight. For each $0\leq k\leq n-1$,  we have
\begin{equation}
\sum_{i=1}^{C_n^{k+1}}\lambda_i(T_{\delta d}^{(k)})\leq\frac{nC_{n-1}^{k}2^{n-k-1}|C_{k+2}(\p\Omega)|}{\min_{1\leq i_1<i_2<\cdots<i_{k+1}\leq n}|C_{i_1i_2\cdots i_{k+1}}(\tilde\Omega)|},
\end{equation}
and
\begin{equation}
\sum_{i=1}^{C_n^{k+1}}\lambda_i(T^{(k)})\leq\frac{nC_{n-1}^{k}2^{n-k-1}|C_{k+2}(\p\Omega)|}{\min_{1\leq i_1<i_2<\cdots<i_{k+1}\leq n}|C_{i_1i_2\cdots i_{k+1}}(\tilde\Omega)|}.
\end{equation}
Here $T_{\delta d}^{(k)}$ and $T^{(k)}$ are the DtN maps for $\tilde\Omega$, $C_{k+2}(\p\Omega)$ means the set of boundary $(k+2)$-cliques in $\tilde\Omega$ and $C_{i_1i_2\cdots i_{k+1}}(\tilde\Omega)$ means the set of $(k+2)$-cliques in $\tilde\Omega$ that are parallel to $\{o,e_{j_1},e_{j_1}+e_{j_2},\cdots,e_{j_1}+e_{j_2}+\cdots+e_{j_{k+1}}\}$ with $(j_1,j_2,\cdots,j_{k+1})$ some permutation of $(i_1,i_2,\cdots,i_{k+1})$.

The rest of the paper is organized as follows. In section 2, we introduce some preliminaries on Hodge theory for graphs. In section 3, we introduce higher order DtN maps on graphs with boundary and derive some of their elementary properties. In section 4, we obtain Raulot-Savo-type estimates for subgraphs of integer lattices. In section 5, we obtain Raulot-Savo-type estimates for subgraphs of standard tessellation of $\R^n$. Finally, in the Appendix, we give the proofs of some facts in the preliminary section and some details of computation in the proof of Theorem \ref{thm-gen}.
\section{Preliminaries}
In this section, we recall some preliminaries for Hodge theory on graphs. In this paper, we assume all graphs $G$ be simple graphs (without loops and multiple edges) with $V(G)$ the set of vertices and $E(G)$ the set of edges.

We first introduce forms and  exterior differentials on graphs. Recall the notion of cliques on a graph.
\begin{defn}
A $k$-clique of a graph $G$ is a subset $\{v_1,v_2,\cdots, v_k\}$  of $V(G)$ with $k$ distinct elements such that $\{v_i,v_j\}\in E(G)$ for any $1\leq i<j\leq k$. The collection of all $k$-cliques of $G$ is denoted as $C_k(G)$. It is clear that $C_1(G)=V(G)$ and $C_2(G)=E(G)$.
\end{defn}
Similar to the smooth case, we define tensors, weights, which play the role of metrics in the smooth case, and forms on graphs.
\begin{defn} Let $G$ be a graph.
\begin{enumerate}
\item A function $f:\underbrace{V(G)\times V(G)\times\cdots\times V(G)}_{k+1}\to \R$ is called a $k$-tensor of $G$ if $f(v_0,v_1,\cdots,v_k)=0$ whenever $\{v_0,v_1,\cdots,v_k\}\not\in C_{k+1}(G)$. The collection of all $k$-tensors on $G$ is denoted as $T^k(G)$. Moreover, let $T^*(G)=\bigoplus_{k=0}^\infty T^k(G)$.
\item A $k$-tensor that is symmetric with respect to all of its variables is called a $k$-weight. The collection of all $k$-weights on $G$ is denoted as $W^k(G)$. Moreover, let $W^*(G)=\bigoplus_{k=0}^\infty W^k(G)$
\item A $k$-weight $w$ with $w(v_0,v_1,\cdots,v_k)>0$ whenever $\{v_0,v_1,\cdots,v_k\}\in C_{k+1}(G)$ is called a positive $k$-weight.
\item A $k$-tensor that is skew symmetric with respect to all of its variables is called a $k$-form. The collection of all $k$-forms on  $G$ is denoted as $A^k(G)$. Moreover, let $A^*(G)=\bigoplus_{k=0}^\infty A^k(G)$.
\end{enumerate}
\end{defn}
As in the smooth case, we define the skew-symmetrization operator as follows.
\begin{defn}
The skew-symmetrization operator $\bA:T^{k}(G)\to A^k(G)$ is defined as
\begin{equation}
\bA(f)(v_0,v_1,\cdots,v_k)=\frac{1}{(k+1)!}\sum_{\sigma\in S_{k+1}}\sgn(\sigma)\cdot f(v_{\sigma(0)},v_{\sigma(1)},\cdots,v_{\sigma(k)})
\end{equation}
where $S_{k+1}$ is the group of permutations on $\{0,1,2,\cdots,k\}$.
\end{defn}
On any graph, a natural weight with every clique of weight $1$ is called the unit weight.
\begin{defn}
We define the unit weight $\mathbf 1_G\in W^*(G)$ on the graph $G$ as
\begin{equation}
\mathbf 1_G(v_0,v_1,\cdots,v_k)=1
\end{equation}
if and only if $\{v_0,v_1,\cdots,v_k\}\in C_{k+1}(G)$ for $k=0,1,2,\cdots$.
\end{defn}
Similar to cup product in algebraic topology, we define tensor product on graphs as follows.
\begin{defn}
Let $G$ be a graph, $f\in T^r(G)$ and $g\in T^s(G)$. We define the tensor product of $f$ and $g$ as
\begin{equation}
\begin{split}
&f\otimes g(v_0,v_1,\cdots,v_{r+s})\\
=&\mathbf 1_G(v_0,v_1,\cdots,v_{r+s})f(v_0,v_1,\cdots,v_r)g(v_r,v_{r+1},\cdots,v_{r+s})
\end{split}
\end{equation}
for any $v_0,v_1,\cdots,v_{r+s}\in V(G)$.
\end{defn}
Similar to the smooth case, we define wedge product of forms on graphs as skew-symmetrization of tensor product.
\begin{defn}
Let $G$ be a graph, $\alpha\in A^r(G)$ and $\beta\in A^s(G)$. Then, define the wedge product of $\alpha$ and $\beta$ as
\begin{equation}
\alpha\wedge\beta=\bA(\alpha\otimes\beta).
\end{equation}
\end{defn}
Similar to the boundary operators in algebraic topology, we define the exterior differential operators on graphs as follows.
\begin{defn}
Let $G$ be a graph and $\alpha\in A^k(G)$. Define the exterior differential of $\alpha$ as
\begin{equation}
\begin{split}
&d\alpha(v_0,v_1,\cdots,v_{k+1})\\
=&\mathbf1_G(v_0,v_1,\cdots,v_{k+1})\sum_{j=0}^{k+1}(-1)^j\alpha(v_0,v_1,\cdots,\hat v_j,\cdots, v_{k+1}).
\end{split}
\end{equation}
for any $v_0,v_1,\cdots,v_{k+1}\in V(G)$.
\end{defn}

Similarly as the smooth case, we have the following conclusions for wedge product and exterior differentials on graphs. For their proofs, see the Appendix.
\begin{prop}\label{prop-exterior}Let $G$ be a graph. Then,
\begin{enumerate}\item $d^2=0$;
\item for any $\alpha\in A^r(G)$ and $\beta\in A^s(G)$, $$\alpha\wedge\beta=(-1)^{rs}\beta\wedge\alpha;$$

\item for any $\alpha\in A^r(G)$ and $\beta\in A^s(G)$, $$d(\alpha\wedge\beta)=d\alpha\wedge\beta+(-1)^r\alpha\wedge d\beta;$$
\item for any $\alpha\in A^p(G),\beta\in A^q(G)$ and $\gamma\in A^r(G)$ with $d\alpha=d\beta=d\gamma=0$,
    $$(\alpha\wedge\beta)\wedge\gamma=\alpha\wedge(\beta\wedge\gamma).$$
\end{enumerate}
\end{prop}
Because of $d^2=0$, we have the following chain complex which is the analogue of the de Rham complex of differential forms:
\begin{equation}
0\to A^{0}(G)\to A^1(G)\to\cdots \to A^r(G)\to A^{r+1}(G)\to\cdots.
\end{equation}
We call this the de Rham complex of the graph $G$. Let $Z^r(G)$ be the kernel of $d:A^r(G)\to A^{r+1}(G)$ and $B^r(G)=d(A^{r-1}(G))$. The cohomology groups:
\begin{equation}
H_{dR}^r(G)=\frac{Z^r(G)}{B^r(G)}\ \ (r=0,1,2,\cdots)
\end{equation}
are called  the de Rham cohomology groups of $G$. The dimension of $H_{dR}^r(G)$ denoted as $b_r(G)$ is called the $r$-th Betti number of $G$. It is clear that $b_0(G)$ is the number of connected components of $G$ when $G$ is a finite graph. In fact, if we associate a finite graph $G$ with a simplical  complex $K(G)$ with each $(k+1)$-clique in $G$ corresponding to a $k$-simplex of $K(G)$ such that a $k$-simplex is a face of a $l$-simplex if and only if the $(k+1)$-clique corresponding to the $k$-simplex is a subset of the $(l+1)$-clique corresponding the $l$-simplex, then $b_i(G)=b_i(K(G))$.

Next, we introduce Hodge theory on graphs.
\begin{defn}
Let $G$ be a graph and $w\in W^*(G)$ be a positive weight on $G$. We call $(G,w)$ a weighted graph.
\end{defn}
A weight on the edges of a graph induces weights on any clique in the following way which is sometimes call the normalized weight.
\begin{defn}
Let $G$ be a graph and $w\in W^1(G)$ be a positive 1-weight on $G$. Then, $w$ induces a positive weight $w^*\in W^*(G)$ on $G$ by
\begin{equation}
w^*(v_0)=\sum_{v_1\in V}w(v_0,v_1)
\end{equation}
for any $v_0\in V(G)$ and
\begin{equation}
w^*(v_0,v_1,\cdots,v_k)=\mathbf 1_G(v_0,v_1,\cdots,v_k)\sum_{0\leq i<j\leq k}w(v_i,v_j)
\end{equation}
for any $v_0,v_1,\cdots,v_k\in V(G)$ and $k=2,3,\cdots$. We simply denote the weighted graph $(G,w^*)$ as $(G,w)$ and call $(G,w)$ an edge-weighted graph.
\end{defn}

Similar to the smooth case, we define inner products for $k$-forms on graphs as follows.
\begin{defn}
Let $(G,w)$ be a finite weighted graph. Define the inner product on $A^k(G)$ as
\begin{equation}
\begin{split}
&\vv<\alpha,\beta>\\
=&\frac{1}{(k+1)!}\sum_{v_0,v_1,\cdots,v_k\in G}\alpha(v_0,v_1,\cdots,v_k)\beta(v_0,v_1,\cdots,v_k)w(v_0,v_1,\cdots,v_k).
\end{split}
\end{equation}
\end{defn}
Let $\delta:A^{k+1}(G)\to A^k(G)$ be the adjoint operator of $d$. By direct computation, one has the following identity.
\begin{prop}[see \cite{Li}]\label{prop-delta}
Let $(G,w)$ be a finite weighted graph. Then
\begin{equation}
\begin{split}
&\delta\alpha(v_0,v_1,\cdots,v_k)\\
=&\frac1{w(v_0,v_1,\cdots,v_k)}\sum_{v\in V(G)}\alpha(v,v_0,v_1,\cdots,v_{k}){w(v,v_0,v_1,\cdots,v_k)}
\end{split}
\end{equation}
for any $\alpha\in A^{k+1}(G)$.
\end{prop}
Let $\Delta=d\delta+\delta d:A^k(G)\to A^k(G)$ which is called the Hodge Laplacian operator on $G$. Forms in the kernel of $\Delta$ are call harmonic forms. The collection of harmonic $k$-forms on $G$ is denoted as $\mathcal H^k(G)$. By using simple linear algebra, it is not hard to see that a discrete version of Hodge decomposition theorem holds for finite weighted graphs.
\begin{thm}
Let $(G,w)$ be a finite weighted graph. Then
\begin{equation}
A^r(G)=d(A^{r-1}(G))\oplus \mathcal H^r(G)\oplus\delta(A^{r+1}(G))
\end{equation}
which is an orthogonal decomposition of $A^r(G)$, for $r=0,1,2,\cdots$. As a consequence,
\begin{equation}
H_{dR}^{r}(G)\cong \mathcal H^r(G)
\end{equation}
for $r=0,1,2,\cdots.$
\end{thm}
\section{Higher order Dirichlet-to-Neumann maps}
In this section, we define higher order DtN maps for graphs with boundary. First, we adopt the notations in \cite{HHW}. Let $G$ be a  graph, and $\Omega$ be a subset of $V(G)$. Let
\begin{equation}
\delta\Omega=\{x\in V(G)\setminus \Omega\ |\ x\ \mbox{is adjacent to some vertices in $\Omega$}. \}
\end{equation}
which is called the vertex boundary of $\Omega$. Let $\ol\Omega=\Omega\cup\delta\Omega$. For any two subsets $A$ and $B$ of $V(G)$, we denote by $E(A,B)$ the set of edges in $G$ with one end in $A$ and the other in $B$. Denote by $\tilde\Omega$ the subgraph of $G$ with $\ol\Omega$ the set of vertices and $E(\Omega,\ol\Omega)$ the set of edges. We will also use $\Omega$ to denote the subgraph of $G$ spanned by $\Omega$. Denote $E(\Omega,\delta\Omega)$ by $\p\Omega$, which is called edge boundary of $\Omega$.

Moreover, let
\begin{equation}
C_k(\p\Omega)=\{\{v_1,v_2,\cdots,v_k\}\in C_k(\tilde\Omega)\ :\ v_1\in \delta\Omega,v_2,\cdots,v_k\in \Omega\}
\end{equation}
which is called the $k$-clique boundary of $\Omega$. It is clear that $C_1(\p\Omega)=\delta\Omega$ and $C_2(\p\Omega)=\p\Omega$. Let
\begin{equation}
\begin{split}
A^{k}(\p\Omega)=\{f:&\delta\Omega\times\underbrace{\Omega\times\cdots\times\Omega}_{k}\to\R\ :\ f(v_0,v_1,v_2,\cdots,v_k)=0\\
 &\mbox{when $\{v_0,v_1,\cdots,v_k\}\not\in C_{k+1}(\p\Omega)$ and $f$ is skew symmetry}\\
 & \mbox{with respect to the last $k$ variables.} \}
\end{split}
\end{equation}
An elements in $A^{k}(\p\Omega)$ is called a boundary $k$-form  of $\Omega$. For a weighted graph, we define the inner product on $A^k(\p\Omega)$ as the follows:
\begin{equation}
\begin{split}
&\vv<\alpha,\beta>_{\p\Omega}\\
=&\frac{1}{k!}\sum_{v\in\delta\Omega}\sum_{v_1,v_2,\cdots,v_k\in\Omega}\alpha(v,v_1,v_2,\cdots,v_k)\beta(v,v_1,v_2,\cdots,v_k)w(v,v_1,v_2,\cdots,v_k).
\end{split}
\end{equation}
More generally, for  $S\subset C_{k+1}(\tilde\Omega)$, we define
\begin{equation}
\vv<\alpha,\beta>_S=\sum_{\{v_0,v_1,\cdots,v_k\}\subset S}\alpha(v_0,v_1,\cdots,v_k)\beta(v_0,v_1,\cdots,v_k)w(v_0,v_1,\cdots,v_k)
\end{equation}
for any $\alpha,\beta\in A^k(\tilde\Omega)$.
It is clear that
\begin{equation}
\vv<\alpha,\beta>_{\tilde\Omega}=\vv<\alpha,\beta>_{\Omega}+\vv<\alpha,\beta>_{\p\Omega}
\end{equation}
for any $\alpha,\beta\in A^k(\tilde\Omega)$.

We also adopt the notion of graphs with boundary in \cite{Pe}. A pair $(G,B)$ is said to be a graph with boundary if
\begin{enumerate}
\item $G$ is a simple graph;
\item $B\subset V(G)$ with $\delta\Omega=B$ and $E(B,B)=\emptyset$ where $\Omega=V(G)\setminus B$ is called the interior of $G$ and  $B$ is called the boundary of $G$.
\end{enumerate}It is clear that $(\tilde \Omega,\delta\Omega)$ in the notations above adopted from \cite{HHW} is a graph with boundary.

For weighted graphs with boundary, we have the following Green's formulas.
\begin{prop}
Let $(G,B,w)$ be a finite weighted graph with  boundary and with $\Omega$ its interior. Then, for any $\alpha\in A^{k+1}(G)$ and $\beta\in A^k(G)$,
\begin{equation}\label{eq-delta}
\vv<\delta\alpha,\beta>_\Omega=\vv<\alpha,d\beta>_{G}-\vv<\mathbf{N}\alpha,\beta>_{\p\Omega}
\end{equation}
and for any $\alpha\in A^{k-1}(G)$ and $\beta\in A^k(G)$
\begin{equation}\label{eq-d}
\vv<d\alpha,\beta>_{\Omega}=\vv<\alpha,\delta\beta>_\Omega-\vv<\mathbf{D}\alpha,\beta>_{\p\Omega}\\
\end{equation}
where $\mathbf{N}\alpha, \mathbf{D}\alpha\in A^{k}(\p\Omega)$ are given by
\begin{equation}\label{eq-def-N}
\mathbf{N}\alpha(v,v_1,v_2,\cdots,v_k)=\frac{\sum_{u\in \Omega}\alpha(u,v,v_1,\cdots,v_k)w(u,v,v_1,\cdots,v_{k})}{w(v,v_1,v_2,\cdots,v_k)}
\end{equation}
and
\begin{equation}\label{eq-def-D}
\mathbf{D}\alpha(v,v_1,v_2,\cdots,v_k)=\alpha(v_1,v_2,\cdots,v_k)
\end{equation}
respectively, when $\{v,v_1,v_2,\cdots,v_k\}\in C_{k+1}(\p\Omega)$.
\end{prop}
\begin{proof} Note that
\begin{equation}
\begin{split}
\vv<\alpha,d\beta>_{G}=\vv<\delta \alpha,\beta>_{G}=\vv<\delta\alpha,\beta>_{\Omega}+\vv<\delta\alpha,\beta>_{\p\Omega}.\\
\end{split}
\end{equation}
Then, by Proposition \ref{prop-delta} and that $E(B,B)=\emptyset$, we get the first identity.

Moreover, by that
\begin{equation}
\begin{split}
\vv<d\alpha,\beta>_{\Omega}+\vv<d\alpha,\beta>_{\p\Omega}=\vv<d\alpha,\beta>_{G}=\vv<\alpha,\delta\beta>_{G}=\vv<\alpha,\delta\beta>_{\Omega}+\vv<\alpha,\delta\beta>_{\p\Omega},\\
\end{split}
\end{equation}
\begin{equation}
\vv<d\alpha,\beta>_{\Omega}=\vv<\alpha,\delta\beta>_{\Omega}+\vv<\alpha,\delta\beta>_{\p\Omega}-\vv<d\alpha,\beta>_{\p\Omega}.
\end{equation}
Note that
\begin{equation}
\begin{split}
&\vv<\alpha,\delta\beta>_{\p\Omega}\\
=&\frac{1}{(k-1)!}\sum_{v\in B,v_1,\cdots,v_{k-1}\in\Omega}\alpha(v,v_1,\cdots,v_{k-1})\delta\beta(v,v_1,\cdots,v_{k-1})w(v,v_1,\cdots,v_{k-1})\\
=&\frac{(-1)^k}{(k-1)!}\sum_{v\in B,v_1,\cdots,v_{k-1},v_k\in\Omega}\alpha(v,v_1,\cdots,v_{k-1})\beta(v,v_1,\cdots,v_k)w(v,v_1,\cdots,v_k)\\
\end{split}
\end{equation}
by Proposition \ref{prop-delta} and $E(B,B)=\emptyset$, and
\begin{equation}
\begin{split}
&\vv<d\alpha,\beta>_{\p\Omega}\\
=&\frac{1}{k!}\sum_{v\in B,v_1,v_2,\cdots,v_k\in\Omega}d\alpha(v,v_1,\cdots,v_k)\beta(v,v_1,\cdots,v_k)w(v,v_1,\cdots,v_k)\\
=&\frac{1}{k!}\sum_{v\in B,v_1,v_2,\cdots,v_k\in\Omega}\alpha(v_1,\cdots,v_k)\beta(v,v_1,\cdots,v_k)w(v,v_1,\cdots,v_k)+\\
&\frac{1}{k!}\sum_{j=1}^k\sum_{v\in B,v_1,v_2,\cdots,v_k\in\Omega}(-1)^j\alpha(v,v_1,\cdots,\hat v_j,\cdots,v_k)\beta(v,v_1,\cdots,v_k)w(v,v_1,\cdots,v_k)\\
=&\vv<\mathbf D\alpha,\beta>_{\p\Omega}+\vv<\alpha,\delta\beta>_{\p\Omega}.
\end{split}
\end{equation}
From these, we get the second identity.
\end{proof}

For the definition of DtN maps, we need to solve some boundary value problems. We are interested in the following three kinds of boundary value problems:
\begin{equation}\label{eq-deltad}
\left\{\begin{array}{ll}
\delta d\omega=0&\mbox{in $\Omega$}\\
\omega=\varphi&\mbox{in $C_{k+1}(\p\Omega)$},
\end{array}\right.
\end{equation}
\begin{equation}\label{eq-ddelta}
\left\{\begin{array}{ll}
d\delta \omega=0&\mbox{in $\Omega$}\\
\omega=\varphi&\mbox{in $C_{k+1}(\p\Omega)$},
\end{array}\right.
\end{equation}
and
\begin{equation}\label{eq-Hodge}
\left\{\begin{array}{ll}
\Delta \omega=0&\mbox{in $\Omega$}\\
\omega=\varphi&\mbox{in $C_{k+1}(\p\Omega)$},
\end{array}\right.
\end{equation}
where $\omega\in A^{k}(G)$ is the unknown $k$-form and $\varphi\in A^k(\p\Omega)$ is the given boundary $k$-form. We will show the solvability of the boundary value problems \eqref{eq-deltad}, \eqref{eq-ddelta} and \eqref{eq-Hodge} by a variation argument. We first need the following simple lemma in linear algebra.
\begin{lem}\label{lem-min}
Let $V$ be a finite dimensional vector space over $\R$ and $Q$ be a nonnegative quadratic form on $V$. Then, for any $x_0\in V$ and any subspace $W\subset V$, the function $f(x)=Q(x_0+x,x_0+x)$ with $x\in W$ achieves its minimum in $W$.
\end{lem}
\begin{proof}
Let $\xi_1,\cdots,\xi_r$ be a basis of $W$ such that
\begin{equation}
Q(\lambda_1\xi_1+\lambda_2\xi_2+\cdots+\lambda_r\xi_r,\lambda_1\xi_1+\lambda_2\xi_2+\cdots+\lambda_r\xi_r)=\sum_{i=1}^rQ(\xi_i,\xi_i)\lambda_i^2.
\end{equation}
Then, for any $x=\lambda_1\xi_1+\lambda_2\xi_2+\cdots+\lambda_r\xi_r$,
\begin{equation}
f(x)=\sum_{i=1}^r\left(Q(\xi_i,\xi_i)\lambda_i^2+2Q(x_0,\xi_i)\lambda_i\right)+Q(x_0,x_0).
\end{equation}
By that $Q$ is nonnegative, we have $Q(\xi_i,\xi_i)\geq 0$ and $Q(x_0,\xi_i)=0$ whenever $Q(\xi_i,\xi_i)=0$. This gives us the conclusion.
\end{proof}

We are now ready to show that the boundary value problems \eqref{eq-deltad}, \eqref{eq-ddelta} and \eqref{eq-Hodge} are all solvable.
\begin{thm}\label{thm-existence}
Let $(G,B,w)$ be a finite weighted graph with boundary and with $\Omega$ its interior. Then, for any given boundary data $\varphi\in A^k(\p\Omega)$, the boundary values problems \eqref{eq-deltad},\eqref{eq-ddelta} and \eqref{eq-Hodge} are all solvable.
\end{thm}
\begin{proof}
We only show the solvability of \eqref{eq-deltad}. The proofs for the other two boundary value problems are similar.

Consider the quadratic form $Q(\alpha,\alpha)=\vv<d\alpha,d\alpha>_{G}$ on $V=A^{k}(G)$. Let
\begin{equation}
W=\{\alpha\in A^{k}(G) |\ \alpha=0\ \mbox{on}\ C_{k+1}(\p\Omega)\}
\end{equation}
and $x_0\in A^{k}(G)$ be the zero-extension of $\varphi$. By Lemma \ref{lem-min}, there is a minimizer of $Q$  when restricting $Q$ on
\begin{equation}
x_0+W=\{\alpha\in A^k(G)\ |\ \alpha=\varphi\ \mbox{on}\ C_{k+1}(\p\Omega)\}.
\end{equation}
Let $\omega$ be such a minimizer. We claim that $\omega$ is a solution to the boundary value problem \eqref{eq-deltad}. Indeed, by that $\omega$ is a minimizer, we know that
\begin{equation}
\vv<d\omega,d\xi>_{G}=0
\end{equation}
for any $\xi\in W$. By \eqref{eq-delta},
\begin{equation}
\vv<\delta d\omega,\xi>_{\Omega}=0
\end{equation}
for any $\xi\in W$. So, $\delta d\omega=0$ in $\Omega$.
\end{proof}

Moreover, the Dirichlet principle holds for the boundary value problems \eqref{eq-deltad}, \eqref{eq-ddelta} and \eqref{eq-Hodge}.
\begin{prop}
\begin{enumerate}
\item The solutions of the boundary value problem \eqref{eq-deltad} are minimizers of the quadratic form $\vv<d\alpha,d\alpha>_{G}$ with $\alpha\in A^k(G)$ and $\alpha=\varphi$ on $C_{k+1}(\p\Omega)$.
\item The solutions of the boundary value problem \eqref{eq-ddelta} are minimizers of the quadratic form $\vv<\delta\alpha,\delta\alpha>_{\Omega}$ with $\alpha\in A^k(G)$ and $\alpha=\varphi$ on $C_{k+1}(\p\Omega)$.
\item The solutions of the boundary value problem \eqref{eq-Hodge} are minimizers of the quadratic form $\vv<d\alpha,d\alpha>_{G}+\vv<\delta\alpha,\delta\alpha>_{\Omega}$ with $\alpha\in A^k(G)$ and $\alpha=\varphi$ on $C_{k+1}(\p\Omega)$.
\end{enumerate}
\end{prop}
\begin{proof}
We only need to show (1). The proofs of (2) and (3) are similar.

Let $\omega_\varphi$ be a solution of the boundary value problem \eqref{eq-deltad} and $\omega\in A^k(G)$ with $\omega=\varphi$ on $C_{k+1}(\p\Omega)$. By \eqref{eq-delta},
\begin{equation}
\vv<d\omega_\varphi,d(\omega-\omega_\varphi)>_{G}=0
\end{equation}
since $\omega-\omega_\varphi=0$ on $C_{k+1}(\p\Omega)$.
So,
\begin{equation}
\vv<d\omega,d\omega>_{G}=\vv<d\omega_\varphi,d\omega_\varphi>_{G}+\vv<d(\omega-\omega_\varphi),d(\omega-\omega_\varphi)>_{G}\geq\vv<d\omega_\varphi,d\omega_\varphi>_{G}.
\end{equation}
This completes the proof.
\end{proof}
We are now ready to introduce DtN maps for forms on graphs. For each $\varphi\in A^k(\p\Omega)$, let $E_{\delta d}^{(k)}(\varphi)$, $E_{d\delta}^{(k)}(\varphi)$ and $E^{(k)}(\varphi)$ be solutions of the boundary value problems \eqref{eq-deltad}, \eqref{eq-ddelta} and \eqref{eq-Hodge} respectively. We define three kinds of DtN maps for forms on graphs corresponding to the three kinds of boundary value problems:
\begin{equation}
T_{\delta d}^{(k)}(\varphi)=\mathbf{N}(dE_{\delta d}^{(k)}(\varphi)),
\end{equation}
\begin{equation}
T_{d\delta }^{(k)}(\varphi)=\mathbf{D}(\delta E_{d\delta }^{(k)}(\varphi))
\end{equation}
and
\begin{equation}
T^{(k)}(\varphi)=\mathbf{N}(dE^{(k)}(\varphi))+\mathbf{D}(\delta E^{(k)}(\varphi)).
\end{equation}
However, because the solutions of the boundary value problems \eqref{eq-deltad}, \eqref{eq-ddelta} and \eqref{eq-Hodge} may not be unique,  we should check that the definitions of the DtN maps $T_{\delta d}^{(k)}$, $T_{d\delta}^{(k)}$ and $T^{k}$ are independent of the choice of solutions of the boundary value problems \eqref{eq-deltad}, \eqref{eq-ddelta} and \eqref{eq-Hodge} respectively.
\begin{thm}\label{thm-well-define}The operators $T_{\delta d}^{(k)}$, $T_{d\delta}^{(k)}$ and $T^{(k)}$ are well defined.
\end{thm}
\begin{proof}
We only need to show the conclusion for $T_{\delta d}^{(k)}$. The proofs for the other two operators are similar.

For $\varphi\in A^k(G)$, let $\omega_1$ and $\omega_2$ be two solutions of the boundary value problem \eqref{eq-deltad}. Then, by \eqref{eq-delta}, we have
\begin{equation}
0=\vv<\delta d(\omega_1-\omega_2),\omega_1-\omega_2>_{\Omega}=\vv<d(\omega_1-\omega_2),d(\omega_1-\omega_2)>_{G}
\end{equation}
since $\omega_1-\omega_2=0$ on $C_{k+1}(\p\Omega)$. So
\begin{equation}
d(\omega_1-\omega_2)=0
\end{equation}
on $G$. Moreover, for any $\xi\in A^k(G)$, by \eqref{eq-delta} again,
\begin{equation}
0=\vv<\delta d(\omega_1-\omega_2),\xi>_{\Omega}=\vv<d(\omega_1-\omega_2),d\xi)>_{G}-\vv<\mathbf{N}d(\omega_1-\omega_2),\xi>_{\p\Omega}.
\end{equation}
So
\begin{equation}
\vv<\mathbf{N}d(\omega_1-\omega_2),\xi>_{\p\Omega}=0
\end{equation}
for any $\xi\in A^k(G)$. This means that
\begin{equation}
\mathbf{N}d\omega_1=\mathbf{N}d\omega_2.
\end{equation}
\end{proof}

Moreover, as usual, the DtN maps just defined are all nonnegative and self-adjoint operators.
\begin{prop}
The DtN maps $T_{\delta d}^{(k)}$, $T_{d\delta}^{(k)}$ and $T^{(k)}$ are all nonnegative self-adjoint operators on $A^{k}(\p\Omega)$. Moreover,
\begin{equation}\label{eq-ker-deltad}
\ker T_{\delta d}^{(k)}=\{\omega|_{C_{k+1}(\p\Omega)}\ :\ \omega\in A^{k}(G)\ \mbox{with}\ d\omega=0 \ \mbox{on}\ G.\},
\end{equation}
\begin{equation}\label{eq-ker-ddelta}
\ker T_{d\delta}^{(k)}=\{\omega|_{C_{k+1}(\p\Omega)}\ :\ \omega\in A^{k}(G)\ \mbox{with}\ \delta\omega=0 \ \mbox{on}\ \Omega.\},
\end{equation}
and
\begin{equation}\label{eq-ker}
\ker T^{(k)}=\{\omega|_{C_{k+1}(\p\Omega)}\ :\ \omega\in A^{k}(G)\ \mbox{with}\ d\omega=0 \ \mbox{on}\ G\ \mbox{and}\ \delta\omega=0 \ \mbox{on}\ \Omega.\}.
\end{equation}
\end{prop}
\begin{proof}
We only need to show the conclusion for  $T_{\delta d}^{(k)}$. The proofs for the other two operators are similar. By \eqref{eq-delta},
\begin{equation}
\vv<T_{\delta d}^{(k)}(\varphi),\psi>_{\p\Omega}=\vv<d\omega_{\varphi},d\omega_{\psi}>_{G}
\end{equation}
for any  $\varphi,\psi\in A^k(\p\Omega)$ where $\omega_{\varphi}$ and $\omega_{\psi}$ mean solutions to the boundary value problem \eqref{eq-deltad} with boundary data $\varphi$ and $\psi$ respectively. From this we get the conclusion.
\end{proof}

It is clear that $T_{\delta d}^{(0)}=T^{(0)}$ is the same as the DtN maps for functions introduced in \cite{HHW} and \cite{HM} for edge -weighted graphs. Moreover, we have the following upper bounds for the norms of the DtN maps that are generalizations of the similar conclusion in \cite{HHW}.

\begin{prop}\label{prop-norm}
Let $(G,B,w)$ be a finite weighted graph with boundary and with $\Omega$ its interior. Then
\begin{equation*}
\|T_{\delta d}^{(k)}\|\leq(k+1)\max_{\{v,u_1,u_2,\cdots,u_k\}\in C_{k+1}(\p\Omega)}\frac{\sum_{u\in\Omega}w(u,v,u_1,\cdots,u_k)}{w(v,u_1,\cdots,u_k)},
\end{equation*}
\begin{equation*}
\|T_{d\delta}^{(k)}\|\leq\max_{\{u_1,u_2,\cdots,u_{k}\}\in C_{k}(\Omega)}\frac{\sum_{v\in B}w(v,u_1,\cdots,u_{k})}{w(u_1,u_2,\cdots,u_k)}
\end{equation*}
for $k\geq 1$, and
\begin{equation*}
\begin{split}
\|T^{(k)}\|\leq&\max_{\{u_1,u_2,\cdots,u_{k}\}\in C_{k}(\Omega)}\frac{\sum_{v\in B}w(v,u_1,\cdots,u_{k})}{w(u_1,u_2,\cdots,u_k)}\\
&+(k+1)\max_{\{v,u_1,u_2,\cdots,u_k\}\in C_{k+1}(\p\Omega)}\frac{\sum_{u\in\Omega}w(u,v,u_1,\cdots,u_k)}{w(v,u_1,\cdots,u_k)}
\end{split}
\end{equation*}
for $k\geq 1$.
\end{prop}
\begin{proof}
Note that for each $\varphi\in A^{k+1}(\p\Omega)$
\begin{equation*}
\begin{split}
&\vv<\bN\varphi,\bN\varphi>_{\p\Omega}\\
=&\frac{1}{k!}\sum_{v\in B}\sum_{u_1,u_2,\cdots,u_k\in \Omega}(\bN\varphi(v,u_1,\cdots,u_k))^2w(v,u_1,\cdots,u_k)\\
=&\frac{1}{k!}\sum_{v\in B}\frac{1}{w(v,u_1,u_2,\cdots,u_k)}\sum_{u_1,u_2,\cdots,u_k\in\Omega}\left(\sum_{u\in\Omega}\varphi(u,v,u_1,\cdots,u_k)w(u,v,u_1,\cdots,u_k)\right)^2\\
\leq&\frac{1}{k!}\sum_{v\in B}\sum_{u_1,\cdots,u_k\in\Omega}\Bigg(\left(\sum_{u\in\Omega}\varphi(v,u,u_1,\cdots,u_k)^2w(v,u,u_1,\cdots,u_k)\right)\\
&\times\left(\sum_{u\in\Omega}\frac{w(u,v,u_1,\cdots,u_k)}{w(v,u_1,\cdots,u_k)}\right)\Bigg)\\
\leq&(k+1)\max_{\{v,u_1,u_2,\cdots,u_k\}\in C_{k+1}(\p\Omega)}\frac{\sum_{u\in\Omega}w(u,v,u_1,\cdots,u_k)}{w(v,u_1,\cdots,u_k)}\vv<\varphi,\varphi>_{\p\Omega}.
\end{split}
\end{equation*}
So, for each $\varphi\in A^{k}(\p\Omega)$,
\begin{equation*}
\begin{split}
&\vv<T_{\delta d}^{(k)}(\varphi),T_{\delta d}^{(k)}(\varphi)>_{\p\Omega}\\
=&\vv<\bN dE_{\delta d}^{(k)}(\varphi),\bN dE_{\delta d}^{(k)}(\varphi)>_{\p\Omega}\\
\leq& (k+1)\max_{\{v,u_1,u_2,\cdots,u_k\}\in C_{k+1}(\p\Omega)}\frac{\sum_{u\in\Omega}w(u,v,u_1,\cdots,u_k)}{w(v,u_1,\cdots,u_k)}\vv<dE_{\delta d}^{(k)}(\varphi),dE_{\delta d}^{(k)}(\varphi)>_{\p\Omega}\\
\leq&(k+1)\max_{\{v,u_1,u_2,\cdots,u_k\}\in C_{k+1}(\p\Omega)}\frac{\sum_{u\in\Omega}w(u,v,u_1,\cdots,u_k)}{w(v,u_1,\cdots,u_k)}\vv<dE_{\delta d}^{(k)}(\varphi),dE_{\delta d}^{(k)}(\varphi)>_{G}\\
\leq&(k+1)\max_{\{v,u_1,u_2,\cdots,u_k\}\in C_{k+1}(\p\Omega)}\frac{\sum_{u\in\Omega}w(u,v,u_1,\cdots,u_k)}{w(v,u_1,\cdots,u_k)}\vv<d\ol\varphi,d\ol\varphi>_{G}\\
=&(k+1)\max_{\{v,u_1,u_2,\cdots,u_k\}\in C_{k+1}(\p\Omega)}\frac{\sum_{u\in\Omega}w(u,v,u_1,\cdots,u_k)}{w(v,u_1,\cdots,u_k)}\vv<d\ol\varphi,d\ol\varphi>_{\p\Omega}\\
\leq&(k+1)^2\left(\max_{\{v,u_1,u_2,\cdots,u_k\}\in C_{k+1}(\p\Omega)}\frac{\sum_{u\in\Omega}w(u,v,u_1,\cdots,u_k)}{w(v,u_1,\cdots,u_k)}\right)^2\vv<\varphi,\varphi>_{\p\Omega}
\end{split}
\end{equation*}
where $\ol\varphi$ means the trivial extension of $\varphi$ and we have used the Dirichlet principle for the boundary value problem \eqref{eq-deltad} and the Cauchy-Schwartz inequality. This give us the first inequality.

Moreover, for any $\varphi\in A^k(\Omega)$,
\begin{equation*}
\begin{split}
&\vv<\bd\varphi,\bd\varphi>_{\p\Omega}\\
=&\frac{1}{(k+1)!}\sum_{v\in B}\sum_{u_1,u_2,\cdots,u_{k+1}\in\Omega}(\bd\varphi(v,u_1,\cdots,u_{k+1}))^2w(v,u_1,\cdots,u_{k+1})\\
=&\frac{1}{(k+1)!}\sum_{u_1,u_2,\cdots,u_{k+1}\in\Omega}\varphi(u_1,\cdots,u_{k+1})^2\sum_{v\in B}w(v,u_1,\cdots,u_{k+1})\\
\leq&\max_{\{u_1,u_2,\cdots,u_{k+1}\}\in C_{k+1}(\Omega)}\frac{\sum_{v\in B}w(v,u_1,\cdots,u_{k+1})}{w(u_1,u_2,\cdots,u_{k+1})}\vv<\varphi,\varphi>_{\Omega}.
\end{split}
\end{equation*}
So, for any $\varphi\in A^k(\p\Omega)$,
\begin{equation*}
\begin{split}
&\vv<T_{d\delta}^{(k)}(\varphi),T_{d\delta}^{(k)}(\varphi)>_{\p\Omega}\\
=&\vv<\bd\delta E_{d\delta}^{(k)}(\varphi),\bd\delta E_{d\delta}^{(k)}(\varphi)>_{\p\Omega}\\
\leq&\max_{\{u_1,u_2,\cdots,u_{k}\}\in C_{k}(\Omega)}\frac{\sum_{v\in B}w(v,u_1,\cdots,u_{k})}{w(u_1,u_2,\cdots,u_k)}\vv<\delta E_{d\delta}^{(k)}(\varphi),\delta E_{d\delta}^{(k)}(\varphi)>_{\Omega}\\
\leq&\max_{\{u_1,u_2,\cdots,u_{k}\}\in C_{k}(\Omega)}\frac{\sum_{v\in B}w(v,u_1,\cdots,u_{k})}{w(u_1,u_2,\cdots,u_k)}\vv<\delta\ol\varphi,\delta \ol\varphi>_{\Omega}\\
\leq&\left(\max_{\{u_1,u_2,\cdots,u_{k}\}\in C_{k}(\Omega)}\frac{\sum_{v\in B}w(v,u_1,\cdots,u_{k})}{w(u_1,u_2,\cdots,u_k)}\right)^2\vv<\varphi,\varphi>_{\p\Omega}
\end{split}
\end{equation*}
where we have used the Dirichlet principle for the boundary value problem \eqref{eq-ddelta} and the Cauchy-Schwartz inequality. This completes the proof of the second inequality.

A combination of the proofs of the first two inequalities will give us the last inequality.
\end{proof}
In the rest of the paper, for a nonnegative self-adjoint liner operator $T:V\to V$ defined on a finite dimensional vector space $V$ with inner product. We will denote by $\lambda_i(T)$ the $i$-th positive eigenvalue of $T$ (in ascending order and counting multiplicities).

As a direct corollary of Proposition \ref{prop-norm}, we have the following upper bounds of the eigenvalues of $T_{\delta d}^{(k)}$, $T_{d\delta}^{(k)}$ and $T^{(k)}$.
\begin{cor}\label{cor-upper}
Let $(G,B,w)$ be a finite weighted graph with boundary and with $\Omega$ its interior. Then,
\begin{equation*}
\lambda_i(T_{\delta d}^{(k)})\leq (k+1)\max_{\{v,u_1,u_2,\cdots,u_k\}\in C_{k+1}(\p\Omega)}\frac{\sum_{u\in\Omega}w(u,v,u_1,\cdots,u_k)}{w(v,u_1,\cdots,u_k)},
\end{equation*}
\begin{equation*}
\lambda_i(T_{d\delta}^{(k)})\leq\max_{\{u_1,u_2,\cdots,u_{k}\}\in C_{k}(\Omega)}\frac{\sum_{v\in B}w(v,u_1,\cdots,u_{k})}{w(u_1,u_2,\cdots,u_k)}
\end{equation*}
for $k\geq 1$,
and
\begin{equation*}
\begin{split}
\lambda_i(T^{(k)})\leq&\max_{\{u_1,u_2,\cdots,u_{k}\}\in C_{k}(\Omega)}\frac{\sum_{v\in B}w(v,u_1,\cdots,u_{k})}{w(u_1,u_2,\cdots,u_k)}\\
&+(k+1)\max_{\{v,u_1,u_2,\cdots,u_k\}\in C_{k+1}(\p\Omega)}\frac{\sum_{u\in\Omega}w(u,v,u_1,\cdots,u_k)}{w(v,u_1,\cdots,u_k)}
\end{split}
\end{equation*}
for $k\geq 1$.
\end{cor}
\section{A Raulot-Savo-type estimate for subgraphs of integer lattices}
In this section, we derive a discrete version of Raulot-Savo-type estimates in \cite{RS1,SY,YY} for finite subgraphs of integer lattices. We first need the following discrete version of Lemma 2.1  in \cite{YY}.
\begin{prop}\label{prop-RS}
Let $(G,B,w)$ be a finite weighted graph with boundary and with $\Omega$ its interior. Let $V$ be a subspace of
\begin{equation}
\{ \xi\in A^{k+1}(G)\ :\ \xi=d\eta\ \mbox{for some}\  \eta\in A^{k}(G)\ \mbox{and}\ \delta\xi=0\ \mbox{on}\ \Omega.\}.
\end{equation}
Suppose $\dim V=n$. Let $A:V\to V$ be the linear transformation on $V$ such that
\begin{equation}\label{eq-A}
\vv<A\xi,\eta>_{G}=\vv<\bN\xi,\bN\eta>_{\p\Omega}
\end{equation}
for any $\xi,\eta\in V$. Then, $A$ is  positive and self-adjoint. Moreover
\begin{equation}\label{eq-deltad-estimate}
\lambda_i(T_{\delta d}^{(k)})\leq \lambda_i(A)
\end{equation}
and
\begin{equation}\label{eq-T-estimate}
\lambda_i(T^{(k)})\leq\lambda_i(A)
\end{equation}
for $i=1,2,\cdots,n$.
\end{prop}
\begin{proof}
By definition of $A$, it is clear that $A$ is self-adjoint. For positivity of $A$, let $\xi\in V$ be such that $\bN\xi=0$. Suppose that $\xi=d\eta$. Then, by Green's formula \eqref{eq-delta},
\begin{equation}
\vv<\xi,\xi>_{G}=\vv<d\eta,\xi>_{G}=\vv<\eta,\delta\xi>_{\Omega}+\vv<\bN\xi,\eta>_{\p\Omega}=0.
\end{equation}
So, $\xi=0$ on $G$ and hence $A$ is positive.

Let $\xi_1,\xi_2,\cdots,\xi_n$ be an orthonormal frame of $V$ such that $A\xi_i=\lambda_i(A)\xi_i$. Suppose that $\xi_i=d\eta_i$. By \eqref{eq-ker-deltad}, without loss of generality,  we can assume that $\eta_i|_{C_{k+1}(\p\Omega)}\perp\ker T_{\delta d}^{(k)}$.

We first claim that $\eta_1|_{C_{k+1}(\p\Omega)},\eta_2|_{C_{k+1}(\p\Omega)},\cdots,\eta_n|_{C_{k+1}(\p\Omega)}$
are linear independent. Indeed, suppose that
\begin{equation}
c_1\eta_1|_{C_{k+1}(\p\Omega)}+c_2\eta_2|_{C_{k+1}(\p\Omega)}+\cdots+c_n\eta_n|_{C_{k+1}(\p\Omega)}=0.
\end{equation}
Let $\eta=c_1\eta_1+c_2\eta_2+\cdots+c_n\eta_n$. Then, by \eqref{eq-delta},
\begin{equation}
\vv<d\eta,d\eta>_{G}=\vv<\delta d\eta,\eta>_{\Omega}+\vv<\bN d\eta,\eta>_{\p\Omega}=0.
\end{equation}
So,
\begin{equation}
\sum_{i=1}^nc_i\xi_i=d\eta=0
\end{equation}
on $G$ which implies that $c_1=c_2=\cdots=c_n=0$.

Let $\varphi_1,\varphi_2,\cdots,\varphi_{N}\in (\ker T_{\delta d}^{(k)})^{\perp}$ be an orthonormal basis of $(\ker T_{\delta d}^{(k)})^{\perp}$ such that $T_{\delta d}^{(k)}(\varphi_i)=\lambda_i(T_{\delta d}^{(k)})\varphi_i$ for $i=1,2,\cdots,N$. Let $$U=\mbox{span}\{\eta_1|_{C_{k+1}(\p\Omega)},\eta_2|_{C_{k+1}(\p\Omega)},\cdots,\eta_i|_{C_{k+1}(\p\Omega)}\}\subset(\ker T_{\delta d}^{(k)})^{\perp}.$$
Then $\dim U=i$. Let $W=\mbox{span}\{\varphi_{i},\varphi_{i+1},\cdots,\varphi_{N}\}$. Then $\dim W=N-i+1$. So, $U\cap W\neq 0$. Let $\varphi\neq0$ be in $U\cap W$. Because $\varphi\in W$, suppose that $\varphi=\sum_{j=i}^N{c_j\varphi_j}$. Then,
\begin{equation}
\begin{split}
\frac{\vv<\bN dE_{\delta d}^{(k)}(\varphi),\bN dE_{\delta d}^{(k)}(\varphi)>_{\p\Omega}}{\vv<dE_{\delta d}^{(k)}(\varphi),dE_{\delta d}^{(k)}(\varphi)>_{G}}=&\frac{\vv<T_{\delta d}^{(k)}(\varphi),T_{\delta d}^{(k)}(\varphi)>_{\p\Omega}}{\vv<T_{\delta d}^{(k)}(\varphi),\varphi>_{\p\Omega}}\\
=&\frac{\sum_{j=i}^Nc_j^2\lambda_j(T_{\delta d}^{(k)})^2}{\sum_{j=i}^Nc_j^2\lambda_j(T_{\delta d}^{(k)})}\\
\geq&\lambda_i(T_{\delta d}^{(k)}).
\end{split}
\end{equation}
On the other hand, because $\varphi\in U$, suppose that $\varphi=\sum_{j=1}^i c_j\eta_j|_{C_{k+1}(\p\Omega)}$. Then,
\begin{equation}
\begin{split}
\frac{\vv<\bN dE_{\delta d}^{(k)}(\varphi),\bN dE_{\delta d}^{(k)}(\varphi)>_{\p\Omega}}{\vv<dE_{\delta d}^{(k)}(\varphi),dE_{\delta d}^{(k)}(\varphi)>_{G}}=&\frac{\vv<\bN d\sum_{j=1}^ic_j\eta_j,\bN d\sum_{j=1}^ic_j\eta_j>_{\p\Omega}}{\vv<d\sum_{j=1}^ic_j\eta_j,d\sum_{j=1}^ic_j\eta_j>_{G}}\\
=&\frac{\vv<\sum_{j=1}^ic_j\bN\xi_j,\sum_{j=1}^ic_j\bN\xi_j>_{\p\Omega}}{\vv<\sum_{j=1}^ic_j\xi_j,\sum_{j=1}^ic_j\xi_j>_G}\\
=&\frac{\sum_{j=1}^ic_j^2\lambda_j(A)}{\sum_{j=1}^ic_j^2}\\
\leq&\lambda_i(A).
\end{split}
\end{equation}
So, we get \eqref{eq-deltad-estimate}.

Furthermore, let $\xi_i$ be the same as before and suppose that $\xi_i=d\eta_i$. Without loss of generality, we assume that $\eta_i\perp B^k(G)$. Then, for any $\alpha\in A^{k-1}(G)$ with $\alpha|_{C_{k}(\p\Omega)}=0$, by \eqref{eq-delta},
\begin{equation}
\vv<\delta\eta_i,\alpha>_{\Omega}=\vv<\eta_i,d\alpha>_{G}-\vv<\mathbf N\eta_i,\alpha>_{\p\Omega}=0.
\end{equation}
So, $\delta\eta_i=0$ on $\Omega$. Let $\ol\eta_i|_{C_{k+1}(\p\Omega)}\in \ker T^{(k)}$ be the orthogonal projection of $\eta_i|_{C_{k+1}(\p\Omega)}$ into $\ker T^{(k)}$ and let $\zeta_i=\eta_i-\ol\eta_i$. Then, by \eqref{eq-ker}, $d\zeta_i=\xi_i$, $\delta \zeta_i=0$ and $\zeta_i\perp \ker T^{(k)}$. By the same argument as before,  $\zeta_1|_{C_{k+1}(\p\Omega)},\zeta_2|_{C_{k+1}(\p\Omega)},\cdots,\zeta_n|_{C_{k+1}(\p\Omega)}$ are linearly independent. From this, a similar argument as before will give us \eqref{eq-T-estimate}.
\end{proof}

We are now ready to prove the main result of this section.
\begin{thm}\label{thm-RS-0}Let $G$ be the graph with $V(G)=\mathbb{Z}^n$ and
\begin{equation}
E(G)=\left\{\{x,y\}\ :\ x,y\in \mathbb Z^n, \sum_{i=1}^n|x_i-y_i|=1\right\}.
\end{equation}
Let $\Omega$ be a nonempty finite subset of $\mathbb Z^n$. Consider the graph $\tilde \Omega$ as an edge-weighted graph with each edge of weight $1$. Then,
\begin{equation}
\lambda_1(T^{(0)})+\lambda_2(T^{(0)})+\cdots+\lambda_n(T^{(0)})\leq \sum_{v\in\delta\Omega}\frac{1}{\deg v}\sum_{i=1}^n\frac{\deg_iv}{|E_i(\tilde\Omega)|}
\end{equation}
where $\deg v$ means the number of edges in $\tilde \Omega$ adjacent to $v$, $\deg_iv$ means the number of edges in $\tilde\Omega$ adjacent to $v$ that are parallel to $e_i$, and $E_i(\tilde\Omega)$ means the set of edges in $\tilde \Omega$ that are parallel to $e_i$. Here $\{e_1,e_2,\cdots,e_n\}$ is the standard
basis of $\R^n$. As a consequence, we have
\begin{equation}
\lambda_1(T^{(0)})+\lambda_2(T^{(0)})+\cdots+\lambda_n(T^{(0)})\leq \frac{|\delta\Omega|}{\min_{1\leq i\leq n}\{|E_i(\tilde\Omega)|\}}.
\end{equation}
\end{thm}
\begin{proof}
Let $x_1,x_2,\cdots,x_n$ be the coordinate functions of $\R^n$. Note that $dx_1,dx_2,\cdots,dx_n\in A^1(\tilde\Omega)$ are linearly independent. Moreover, for each $x_0\in\Omega$,
\begin{equation}
\delta dx_i(x_0)=\frac{1}{2n}\sum_{j=1}^n(2x_i(x_0)-x_i(x_0+e_j)-x_i(x_0-e_j))=0.
\end{equation}
Let $V=\mbox{span}\{dx_1,dx_2,\cdots,dx_n\}$ and let $A$ be the linear transformation defined in \eqref{eq-A}. Then, by Proposition \ref{prop-RS}, we know that
\begin{equation}
\sum_{i=1}^n\lambda_i(T^{(0)})\leq \tr A.
\end{equation}
Moreover, note that
\begin{equation}
\vv<dx_i,dx_j>_{\tilde\Omega}=\left\{\begin{array}{ll}0& i\neq j\\|E_i(\tilde\Omega)|& i=j
\end{array}\right.
\end{equation}
and
\begin{equation}
\begin{split}
\vv<\bN dx_i,\bN dx_i>_{\p\Omega}=\sum_{v\in\delta\Omega}\frac{1}{\deg v}\left(\sum_{u\in\Omega,u\sim v}(x_i(v)-x_i(u))\right)^2\leq\sum_{v\in\delta\Omega}\frac{\deg_iv}{\deg v}.
\end{split}
\end{equation}
So,
\begin{equation}
\tr A=\sum_{i=1}^n\frac{\vv<\bN dx_i,\bN dx_i>_{\p\Omega}}{\vv<dx_i,dx_i>_{\tilde\Omega}}\leq\sum_{v\in\delta\Omega}\frac{1}{\deg v}\sum_{i=1}^n\frac{\deg_iv}{|E_i(\tilde\Omega)|}.
\end{equation}
This completes the proof of Theorem \ref{thm-RS-0}
\end{proof}
\section{Raulot-Savo-type estimates for subgraphs of the standard tessellation of $\R^n$}
In this section, we want to make some applications of Proposition \ref{prop-RS} to higher order forms. Because integer lattices do not contain any triangle, we consider the graph $G$ of standard tessellation of $\R^n$ with $V(G)=\mathbb Z^n$ and
\begin{equation}
E(G)=\left\{\{x,y\}\ :\ x\neq y\in \mathbb Z^n, 0\leq \min_{1\leq i\leq n}(y_i-x_i)\leq \max_{1\leq i\leq n}(y_i-x_i)=1\right\}.
\end{equation}

For simplicity, we first consider estimates for eigenvalues of DtN maps for functions.
\begin{thm}\label{thm-zero}
Let $G$ be the standard tessellation of $\R^n$ and $\Omega$ be a nonempty finite subset of $\mathbb Z^n$. Consider $\tilde\Omega$ as a weighted graph with unit weight. Then
\begin{equation}
\begin{split}
\sum_{i=1}^n\lambda_i(T^{(0)})\leq2^{n-1}\sum_{i=1}^n\frac{1}{|E_i(\tilde\Omega)|}\sum_{k=1}^n\sum_{\tiny{\begin{array}{l}1\leq i_1<i_2<\cdots<i_k\leq n\\ i\in \{i_1,i_2,\cdots,i_k\}\end{array}}}|E_{i_1i_2\cdots i_k}(\p\Omega)|
\end{split}
\end{equation}
where  $E_{i_1i_2\cdots i_k}(\p\Omega)$ means the number of edges in $\p\Omega$ that are parallel to $e_{i_1}+e_{i_2}+\cdots+e_{i_k}$. As a consequence,
\begin{equation}
\sum_{i=1}^n\lambda_i(T^{(0)})\leq \frac{n2^{n-1}|\p\Omega|}{\min_{1\leq i\leq n}\{|E_i(\tilde\Omega)|\}}.
\end{equation}
\end{thm}
\begin{proof}
Consider $V=\mbox{span}\{dx_1,dx_2,\cdots,dx_n\}$, it is not hard to see that $dx_1,dx_2,\cdots,dx_n$ are linearly independent and we still have
\begin{equation}
\delta dx_i(u)=0
\end{equation}
for any $u\in \Omega$. Moreover,
\begin{equation}
\vv<dx_i,dx_j>_{\tilde\Omega}=\sum_{k=1}^n\sum_{\tiny{\begin{array}{l}1\leq i_1<\cdots<i_k\leq n\\\{i,j\}\subset\{i_1,i_2,\cdots,i_k\}\end{array}}}|E_{i_1i_2\cdots i_k}(\tilde\Omega)|,
\end{equation}
and
\begin{equation}
\begin{split}
\vv<\mathbf Ndx_i,\mathbf Ndx_i>_{\p\Omega}=&\sum_{v\in\delta\Omega}\left(\sum_{u\in\Omega,u\sim v}(v_i-u_i)\right)^2\\
\leq&2^{n-1}\sum_{v\in\delta\Omega}\sum_{k=1}^n\sum_{\tiny{\begin{array}{l}1\leq i_1<i_2<\cdots<i_k\leq n\\ i\in \{i_1,i_2,\cdots,i_k\}\end{array}}}\deg_{i_1i_2\cdots i_k}v\\
=&2^{n-1}\sum_{k=1}^n\sum_{\tiny{\begin{array}{l}1\leq i_1<i_2<\cdots<i_k\leq n\\ i\in \{i_1,i_2,\cdots,i_k\}\end{array}}}|E_{i_1i_2\cdots i_k}(\p\Omega)|\\
\end{split}
\end{equation}
for $i=1,2,\cdots,n$, where $\deg_{i_1i_2\cdots i_k}v$ means the number of edges adjacent to $v$ that are parallel to $e_{i_1}+e_{i_2}+\cdots+e_{i_k}$. The factor $2^{n-1}$ in the last inequality comes from the fact that there are at most $2^{n-1}$ vertices $u$ adjacent to $v$ in $\tilde\Omega$ such that $v_i-u_i$ all equal to $1$ or all equal to $-1$.

Let $B=\left(\vv<dx_i,dx_j>_{\tilde\Omega}\right)_{i,j=1,2,\cdots,n}$ and $C=\left(\vv<\mathbf Ndx_i,\mathbf Ndx_j>_{\p\Omega}\right)_{i,j=1,2,\cdots,n}$. We claim that
\begin{equation}
B\geq D:=\mbox{diag}\{|E_1(\tilde\Omega)|,|E_2(\tilde\Omega)|,\cdots, |E_n(\tilde\Omega)|\}.
\end{equation}
In fact, for any $(\xi_1,\xi_2,\cdots,\xi_n)\in \R^n$.
\begin{equation}
\begin{split}
&\sum_{i,j=1}^n(B-D)_{ij}\xi_i\xi_j\\
=&\sum_{k=2}^{n}\sum_{i,j=1}^n\sum_{\tiny{\begin{array}{l}1\leq i_1<\cdots<i_k\leq n\\\{i,j\}\subset\{i_1,i_2,\cdots,i_k\}\end{array}}}|E_{i_1i_2\cdots i_k}(\tilde\Omega)|\xi_i\xi_j\\
=&\sum_{k=2}^{n}\sum_{1\leq i_1<i_2<\cdots\leq i_k\leq n}|E_{i_1i_2\cdots i_k}(\tilde\Omega)|(\xi_{i_1}+\xi_{i_2}+\cdots+\xi_{i_k})^2\\
\geq&0.
\end{split}
\end{equation}
Then, by Proposition \ref{prop-RS},
\begin{equation*}
\begin{split}
&\sum_{i=1}^n\lambda_i(T^{(0)})\\
\leq&\tr A\\
=&\tr(B^{-1}C)\\
\leq&\tr(D^{-1}C)\\
=&2^{n-1}\sum_{i=1}^n\frac{1}{|E_i(\tilde\Omega)|}\sum_{k=1}^n\sum_{\tiny{\begin{array}{l}1\leq i_1<i_2<\cdots<i_k\leq n\\ i\in \{i_1,i_2,\cdots,i_k\}\end{array}}}|E_{i_1i_2\cdots i_k}(\p\Omega)|\\
\leq&\frac{2^{n-1}}{\min_{1\leq i\leq n}\{|E_i(\tilde\Omega)|\}}\sum_{i=1}^n\sum_{k=1}^n\sum_{\tiny{\begin{array}{l}1\leq i_1<i_2<\cdots<i_k\leq n\\ i\in \{i_1,i_2,\cdots,i_k\}\end{array}}}|E_{i_1i_2\cdots i_k}(\p\Omega)|\\
=&\frac{2^{n-1}}{\min_{1\leq i\leq n}\{|E_i(\tilde\Omega)|\}}\sum_{k=1}^n\sum_{1\leq i_1<i_2<\cdots<i_k\leq n}k|E_{i_1i_2\cdots i_k}(\p\Omega)|\\
\leq& \frac{n2^{n-1}|\p\Omega|}{\min_{1\leq i\leq n}\{|E_i(\tilde\Omega)|\}}.
\end{split}
\end{equation*}
This completes the proof of the theorem.
\end{proof}

More generally, we have the following estimate.
\begin{thm}\label{thm-gen}
Let $G$ be the standard  tessellation of $\R^n$ and $\Omega$ be a nonempty finite subset of $\mathbb Z^n$ and $0\leq k\leq n-1$.  Consider $\tilde\Omega$ as a weighted graph with unit weight and suppose that $C_{i_1i_2\cdots i_{k+1}}(\tilde\Omega)\neq\emptyset$ for any $1\leq i_1< i_2<\cdots<i_{k+1}\leq n$. Here $C_{i_1i_2\cdots i_{k+1}}(\tilde\Omega)$ means the set of $(k+2)$-cliques in $\tilde\Omega$ such that the $(k+1)$-simplex formed by the $(k+2)$ vertices of the $(k+2)$-clique is parallel to the $(k+1)$-simplex formed by $o,e_{j_1},e_{j_1}+e_{j_2},\cdots,e_{j_1}+\cdots +e_{j_{k+1}}$ with $(j_1,j_2,\cdots,j_{k+1})$ some permutation of $(i_1,i_2,\cdots,i_{k+1})$. Then,
\begin{equation}
\sum_{i=1}^{C_n^{k+1}}\lambda_i(T_{\delta d}^{(k)})\leq (k+1)2^{n-k-1}\sum_{1\leq i_1<i_2<\cdots<i_{k+1}\leq n}\frac{|C_{k+2}(\p\Omega)|}{|C_{i_1i_2\cdots i_{k+1}}(\tilde\Omega)|},
\end{equation}
and
\begin{equation}
\sum_{i=1}^{C_n^{k+1}}\lambda_i(T^{(k)})\leq (k+1)2^{n-k-1}\sum_{1\leq i_1<i_2<\cdots<i_{k+1}\leq n}\frac{|C_{k+2}(\p\Omega)|}{|C_{i_1i_2\cdots i_{k+1}}(\tilde\Omega)|}.
\end{equation}
As a consequence, we have
\begin{equation}
\sum_{i=1}^{C_n^{k+1}}\lambda_i(T_{\delta d}^{(k)})\leq\frac{nC_{n-1}^{k}2^{n-k-1}|C_{k+2}(\p\Omega)|}{\min_{1\leq i_1<i_2<\cdots<i_{k+1}\leq n}|C_{i_1i_2\cdots i_{k+1}}(\tilde\Omega)|},
\end{equation}
and
\begin{equation}
\sum_{i=1}^{C_n^{k+1}}\lambda_i(T^{(k)})\leq\frac{nC_{n-1}^{k}2^{n-k-1}|C_{k+2}(\p\Omega)|}{\min_{1\leq i_1<i_2<\cdots<i_{k+1}\leq n}|C_{i_1i_2\cdots i_{k+1}}(\tilde\Omega)|}.
\end{equation}
\end{thm}
\begin{proof}
Let
\begin{equation}
V=\mbox{span}\{dx_{i_1}\wedge dx_{i_2}\wedge\cdots\wedge dx_{i_{k+1}}:1\leq i_1<i_2<\cdots<i_{k+1}\leq n\}.
 \end{equation}
 Let
 \begin{equation}
 C^*_{k+2}(\tilde\Omega)=\cup_{1\leq i_1<i_2<\cdots<i_{k+1}\leq n}C_{i_1i_2\cdots i_{k+1}}(\tilde\Omega)
  \end{equation}
  and $C'_{k+2}(\tilde\Omega)=C_{k+2}(\tilde\Omega)\setminus C^*_{k+2}(\tilde\Omega)$.
Note that, by Proposition \ref{prop-em} in the Appendix,
\begin{equation}\label{eq-0}
dx_{i_1}\wedge dx_{i_2}\wedge\cdots\wedge dx_{i_{k+1}}(u_1,u_2,\cdots,u_{k+2})=0
\end{equation}
for any $\{u_1,u_2,\cdots,u_{k+2}\}\in C^*_{k+2}(\tilde\Omega)\setminus C_{i_1i_2\cdots i_{k+1}}(\tilde\Omega)$ and
\begin{equation}\label{eq-not-0}
dx_{i_1}\wedge dx_{i_2}\wedge\cdots\wedge dx_{i_{k+1}}(u_1,u_2,\cdots,u_{k+2})=\pm \frac{1}{(k+1)!}
\end{equation}
for any $\{u_1,u_2,\cdots,u_{k+2}\}\in C_{i_1i_2\cdots i_{k+1}}(\tilde\Omega)$. From this and by that $C_{i_1i_2\cdots i_{k+1}}(\tilde\Omega)\neq\emptyset$ for any $1\leq i_1<i_2<\cdots<i_{k+1}\leq n$, $\dim V=C_{n}^{k+1}$.

Moreover, by Proposition \ref{prop-compute} in the Appendix,
\begin{equation}
\delta (dx_{i_1}\wedge dx_{i_2}\wedge\cdots\wedge dx_{i_{k+1}})=0
\end{equation}
in $\Omega$. Furthermore, by \eqref{eq-0} and \eqref{eq-not-0},
\begin{equation*}
\begin{split}
&\vv<dx_{i_1}\wedge\cdots dx_{i_{k+1}},dx_{j_1}\wedge\cdots\wedge dx_{j_{k+1}}>_{\tilde\Omega}\\
=&\vv<dx_{i_1}\wedge\cdots dx_{i_{k+1}},dx_{j_1}\wedge\cdots\wedge dx_{j_{k+1}}>_{C^*_{k+2}(\tilde\Omega)}\\
&+\vv<dx_{i_1}\wedge\cdots dx_{i_{k+1}},dx_{j_1}\wedge\cdots\wedge dx_{j_{k+1}}>_{C'_{k+2}(\tilde\Omega)}\\
=&\frac{1}{(k+1)!^2}|C_{i_1i_2\cdots i_{k+1}}(\tilde\Omega)|\delta_{i_1i_2\cdots i_{k+1}}^{j_1j_2\cdots j_{k+1}}+\vv<dx_{i_1}\wedge\cdots dx_{i_{k+1}},dx_{j_1}\wedge\cdots\wedge dx_{j_{k+1}}>_{C'_{k+2}(\tilde\Omega)}.\\
\end{split}
\end{equation*}
On the other hand,
\begin{equation*}
\begin{split}
&\vv<\bN dx_{i_1}\wedge\cdots\wedge dx_{i_{k+1}},\bN dx_{i_1}\wedge\cdots\wedge dx_{i_{k+1}}>_{\p\Omega}\\
=&\sum_{\{v,u_1,u_2,\cdots,u_{k}\}\in C_{k+1}(\p\Omega)}\left(\sum_{u\in\Omega}dx_{i_1}\wedge\cdots\wedge dx_{i_{k+1}}(u,v,u_1,\cdots,u_{k+1})\right)^2\\
\leq& \frac{2^{n-k-1}}{(k+1)!^2}\sum_{\{v,u_1,u_2,\cdots,u_{k}\}\in C_{k+1}(\p\Omega)}|C_{k+2}(v,u_1,\cdots,u_{k})|\\
=&\frac{(k+1)2^{n-k-1}}{(k+1)!^2}|C_{k+2}(\p\Omega)|
\end{split}
\end{equation*}
where $C_{k+2}(v,u_1,\cdots,u_{k})$ means the set of $(k+2)$-cliques in $\tilde\Omega$ adjacent to $\{v,u_1,\cdots,u_k\}$. The inequality above comes from the fact that $dx_{i_1}\wedge\cdots\wedge dx_{i_{k+1}}(u,v,u_1,\cdots,u_{k+1})$ only takes values $\frac{1}{(k+1)!}$, $-\frac{1}{(k+1)!}$ or $0$ and the number of $(k+2)$-cliques $\{u,v,u_1,\cdots,u_{k+1}\}$ in $\tilde\Omega$ adjacent to $\{v,u_1,\cdots,u_k\}$ such that $dx_{i_1}\wedge\cdots\wedge dx_{i_{k+1}}(u,v,u_1,\cdots,u_{k+1})$ all equal to $\frac{1}{(k+1)!}$ or all equal to $-\frac{1}{(k+1)!}$ is at most $2^{n-k-1}$ (see Proposition \ref{prop-compute} in the Appendix).

Finally, as in the proof of Theorem \ref{thm-zero}, by Proposition \ref{prop-RS},
\begin{equation*}
\begin{split}
\sum_{i=1}^{C_n^{k+1}}\lambda_i(T_{\delta d}^{(k)}),\sum_{i=1}^{C_n^{k+1}}\lambda_i(T^{(k)})\leq{(k+1)2^{n-k-1}}\sum_{1\leq i_1<i_2<\cdots<i_{k+1}\leq n}\frac{|C_{k+2}(\p\Omega)|}{|C_{i_1i_2\cdots i_{k+1}}(\tilde\Omega)|}.
\end{split}
\end{equation*}
\end{proof}
\section{Appendix}
In this section, we give the proof of Proposition \ref{prop-exterior} and some details of computation in the proof of Theorem \ref{thm-gen}.

 Although the definitions in Section 2 and the facts in Proposition \ref{prop-exterior} are very similar to that of exterior calculus for simplicial complexes. In the first part of this section, we present the proof of Proposition \ref{prop-exterior} because our notations and definitions are slightly different with that for simplicial complexes.
\begin{proof}[Proof of Proposition \ref{prop-exterior}](1) It is just by simple direct computation.\\
(2) For any $\{v_0,v_1,\cdots,v_{r+s}\}\in C_{r+s+1}(G)$,
\begin{equation*}
\begin{split}
&(r+s+1)!\alpha\wedge\beta(v_0,v_1,\cdots,v_{r+s})\\
=&\sum_{\sigma\in S_{r+s+1}}\sgn(\sigma)\alpha(v_{\sigma(0)},v_{\sigma(1)},\cdots,v_{\sigma(r)})\beta(v_{\sigma(r)},v_{\sigma(r+2)},\cdots,v_{\sigma(r+s)})\\
=&\sum_{\sigma\in S_{r+s+1}}\sgn(\sigma)(-1)^\frac{r^2+r+s^2+s}{2}\beta(v_{\sigma(r+s)},v_{\sigma(r+s-1)},\cdots,v_{\sigma(r)})\alpha(v_{\sigma(r)},v_{\sigma(r-1)},\cdots,v_{\sigma(0)})\\
=&\sum_{\sigma\in S_{r+s+1}}\sgn(\sigma)(-1)^\frac{r^2+r+s^2+s}{2}\beta(v_{\sigma(\tau(0))},v_{\sigma(\tau(1))},\cdots,v_{\sigma(\tau(s))})\\
&\times\alpha(v_{\sigma(\tau(s))},v_{\sigma(\tau(s+1))},\cdots,v_{\sigma(\tau(r+s))})\\
=&\sum_{\sigma\in S_{r+s+1}}\sgn(\sigma\cdot\tau^{-1})(-1)^\frac{r^2+r+s^2+s}{2}\beta(v_{\sigma(0)},v_{\sigma(1)},\cdots,v_{\sigma(s)})\alpha(v_{\sigma(s)},v_{\sigma(s+1)},\cdots,v_{\sigma(r+s)})\\
=&\sum_{\sigma\in S_{r+s+1}}\sgn(\sigma)(-1)^\frac{r^2+r+s^2+s+(r+s)^2+(r+s)}{2}\beta(v_{\sigma(0)},v_{\sigma(1)},\cdots,v_{\sigma(s)})\\
&\times\alpha(v_{\sigma(s)},v_{\sigma(s+1)},\cdots,v_{\sigma(r+s)})\\
=&(-1)^{rs}(r+s+1)!\beta\wedge\alpha(v_0,v_1,\cdots,v_{r+s})
\end{split}
\end{equation*}
where $\tau=\left(\begin{array}{cccccc}0&1&2&\cdots&r+s-1&r+s\\r+s&r+s-1&r+s-2&\cdots&1&0
\end{array}\right).$\\
(3) For any  $\{v_0,v_1,\cdots,v_{r+s+1}\}\in C_{r+s+2}(G)$,
\begin{equation*}
\begin{split}
&d(\alpha\wedge\beta)(v_0,v_1,\cdots,v_{r+s+1})\\
=&\sum_{j=0}^{r+s+1}(-1)^j\alpha\wedge\beta(v_0,v_1,\cdots,\hat v_j,\cdots,v_{r+s+1})\\
=&\frac{1}{(r+s+1)!}\sum_{j=0}^{r+s+1}(-1)^j\sum_{\sigma\in S_{r+s+2}^j}\sgn(\sigma)\alpha\otimes\beta(v_{\sigma(0)},v_{\sigma(1)},\cdots,\hat v_j,\cdots,v_{\sigma(r+s+1)})\\
\end{split}
\end{equation*}
where $S_{r+s+2}^j=\{\sigma\in S_{r+s+2}\ |\ \sigma(j)=j\}.$
Moreover,
\begin{equation*}
\begin{split}
&d\alpha\wedge\beta+(-1)^r\alpha\wedge d\beta(v_0,v_1,\cdots,v_{r+s+1})\\
=&\frac{1}{(r+s+2)!}\sum_{\sigma\in S_{r+s+2}}\sgn\sigma d\alpha(v_{\sigma(0)},v_{\sigma(1)},\cdots,v_{\sigma(r+1)})\beta(v_{\sigma(r+1)},v_{\sigma(r+2)},\cdots,v_{\sigma(r+s+1)})+\\
&\frac{(-1)^r}{(r+s+2)!}\sum_{\sigma\in S_{r+s+2}}\sgn\sigma\alpha(v_{\sigma(0)},v_{\sigma(1)},\cdots,v_{\sigma(r)})d\beta(v_{\sigma(r)},v_{\sigma(r+1)},v_{\sigma(r+2)},\cdots,v_{\sigma(r+s+1)})\\
=&I_1+I_2+I_3+I_4\\
\end{split}
\end{equation*}
where
\begin{equation}\label{eq-I1}
\begin{split}
I_1=&\frac{1}{(r+s+2)!}\sum_{\sigma\in S_{r+s+2}}\sgn\sigma\sum_{j=0}^{r}(-1)^j\alpha(v_{\sigma(0)},\cdots,\hat v_{\sigma(j)},\cdots,v_{\sigma(r+1)})\\
&\times\beta(v_{\sigma(r+1)},v_{\sigma(r+2)},\cdots,v_{\sigma(r+s+1)})\\
=&\frac{1}{(r+s+2)!}\sum_{\sigma\in S_{r+s+2}}\sgn\sigma\sum_{j=0}^{r}(-1)^j\alpha\otimes\beta(v_{\sigma(0)},\cdots,\hat v_{\sigma(j)},\cdots,v_{\sigma(r+s+1)})\\
=&\frac{1}{(r+s+2)!}\sum_{j=0}^{r}(-1)^j\sum_{i=0}^{r+s+1}\sum_{\sigma\in S^{j,i}_{r+s+2}}\sgn\sigma\alpha\otimes\beta(v_{\sigma(0)},\cdots,\hat v_{\sigma(j)},\cdots,v_{\sigma(r+s+1)})\\
=&\frac{1}{(r+s+2)!}\sum_{j=0}^{r}(-1)^j\sum_{i=0}^{r+s+1}\sum_{\tau\in S^{i}_{r+s+2}}(-1)^{j-i}\sgn\tau\alpha\otimes\beta(v_{\tau(0)},\cdots,\hat v_{\tau(i)},\cdots,v_{\tau(r+s+1)})\\
=&\frac{r+1}{(r+s+2)!}\sum_{i=0}^{r+s+1}(-1)^{i}\sum_{\tau\in S^{i}_{r+s+2}}\sgn\tau\alpha\otimes\beta(v_{\tau(0)},\cdots,\hat v_{\tau(i)},\cdots,v_{\tau(r+s+1)})\\
\end{split}
\end{equation}
\begin{equation}
\begin{split}
I_2=\frac{(-1)^{r+1}}{(r+s+2)!}\sum_{\sigma\in S_{r+s+2}}\sgn\sigma\alpha(v_{\sigma(0)},\cdots,v_{\sigma(r)})\beta(v_{\sigma(r+1)},v_{\sigma(r+2)},\cdots,v_{\sigma(r+s+1)}),
\end{split}
\end{equation}
\begin{equation}
\begin{split}
I_3=&\frac{(-1)^r}{(r+s+2)!}\sum_{\sigma\in S_{r+s+2}}\sgn\sigma\alpha(v_{\sigma(0)},\cdots,v_{\sigma(r)})\beta(v_{\sigma(r+1)},v_{\sigma(r+2)},\cdots,v_{\sigma(r+s+1)})\\
=&-I_2,
\end{split}
\end{equation}
and similarly as in \eqref{eq-I1},
\begin{equation}
\begin{split}
I_4=&\frac{1}{(r+s+2)!}\sum_{\sigma\in S_{r+s+2}}\sgn\sigma\sum_{j=r+1}^{r+s+1}(-1)^j\alpha(v_{\sigma(0)},v_{\sigma(1)},\cdots,v_{\sigma(r)})\\
&\times\beta(v_{\sigma(r)},\cdots,\hat v_{\sigma(j)},\cdots,v_{\sigma(r+s+1)})\\
=&\frac{s+1}{(r+s+2)!}\sum_{i=0}^{r+s+1}(-1)^{i}\sum_{\tau\in S^{i}_{r+s+2}}\sgn\tau\alpha\otimes\beta(v_{\tau(0)},\cdots,\hat v_{\tau(i)},\cdots,v_{\tau(r+s+1)})\\
\end{split}
\end{equation}
with $S_{r+s+2}^{j,i}=\{\sigma\in S_{r+s+2}\ |\ \sigma(j)=i\}$. So,
\begin{equation}
d(\alpha\wedge\beta)=d\alpha\wedge\beta+(-1)^r\alpha\wedge d\beta.
\end{equation}
(4) For any $\{u_0,u_1,\cdots,u_{p+q+r}\}\in C_{p+q+r+1}(G)$,
\begin{equation*}
\begin{split}
&(\alpha\wedge\beta)\wedge\gamma(u_0,u_1,\cdots,u_{p+q+r})\\
=&\frac{1}{(p+q+r+1)!}\sum_{\sigma\in S_{p+q+r+1}}\sgn\sigma\alpha\wedge\beta(u_{\sigma(0)},\cdots,u_{\sigma(p+q)})\gamma(u_{\sigma(p+q)},\cdots,u_{\sigma(p+q+r)})\\
=&\frac{1}{(p+q+r+1)!(p+q+1)!}\sum_{\sigma\in S_{p+q+r+1}}\sgn\sigma\sum_{\tau\in S_{p+q+1}}\sgn\tau\alpha(u_{\sigma(\tau(0))},\cdots,u_{\sigma(\tau(p))})\\
&\times\beta(u_{\sigma(\tau(p))},\cdots,u_{\sigma(\tau(p+q))})\gamma(u_{\sigma(p+q)},\cdots,u_{\sigma(p+q+r)})\\
=&\frac{1}{(p+q+r+1)!(p+q+1)!}(I_1+I_2+I_3).
\end{split}
\end{equation*}
Here
\begin{equation*}
\begin{split}
I_1=&\sum_{i=0}^{p-1}\sum_{\sigma\in S_{p+q+r+1}}\sgn\sigma\sum_{\tau\in S_{p+q+1}^{i,p+q}}\sgn\tau\alpha(u_{\sigma(\tau(0))},\cdots,u_{\sigma(\tau(p))})\beta(u_{\sigma(\tau(p))},\cdots,u_{\sigma(\tau(p+q))})\\
&\times\gamma(u_{\sigma(p+q)},\cdots,u_{\sigma(p+q+r)})\\
=&\sum_{i=0}^{p-1}\sum_{\sigma\in S_{p+q+r+1}}\sgn\sigma\sum_{\tau\in S_{p+q+1}^{i,p+q}}\sgn\tau\sum_{j=0}^{p}(-1)^{p-j}\alpha(u_{\sigma(\tau(0))},\cdots,\widehat{u_{\sigma(\tau(j))}},\cdots,u_{\sigma(\tau(p+1))})\times\\
&\sum_{k=0}^q(-1)^{q-k}\beta(u_{\sigma(\tau(p))},\cdots,\widehat{u_{\sigma(\tau(p+k))}},\cdots,u_{\sigma(\tau(p+q))},u_{\sigma(\tau(i))})\gamma(u_{\sigma(p+q)},\cdots,u_{\sigma(p+q+r)})\\
=& J_1+J_2+J_3
\end{split}
\end{equation*}
where we have used the fact that $d\alpha=d\beta=0$, and
\begin{equation*}
\begin{split}
J_1=&\sum_{i=0}^{p-1}\sum_{\sigma\in S_{p+q+r+1}}\sgn\sigma\sum_{\tau\in S_{p+q+1}^{i,p+q}}(-1)^{p+q-i}\sgn\tau\\ &\times\alpha(u_{\sigma(\tau(0))},\cdots,\widehat{u_{\sigma(\tau(i))}},\cdots,u_{\sigma(\tau(p))},u_{\sigma(\tau(p+1))})\\
&\times\beta(u_{\sigma(\tau(p+1))},\cdots,u_{\sigma(\tau(p+q))},u_{\sigma(\tau(i))})\gamma(u_{\sigma(p+q)},\cdots,u_{\sigma(p+q+r)})\\
=&\sum_{i=0}^{p-1}\sum_{\sigma\in S_{p+q+r+1}}\sgn\sigma\sum_{\mu\in S_{p+q+1}^{p+q}}\sgn\mu \alpha(u_{\sigma(\mu(0))},\cdots,u_{\sigma(\mu(p))})\beta(u_{\sigma(\mu(p))},\cdots,u_{\sigma(\mu(p+q))})\\
&\times\gamma(u_{\sigma(p+q)},\cdots,u_{\sigma(p+q+r)})\\
\end{split}
\end{equation*}
with the relation of $\mu$ and $\tau$ given by $\mu(0)=\tau(0),\cdots, \mu(i-1)=\tau(i-1),\mu(i)=\tau(i+1),\cdots,\mu(p+q-1)=\tau(p+q),\mu(p+q)=\tau(i)=p+q$,
\begin{equation*}
\begin{split}
J_2=&\sum_{i=0}^{p-1}\sum_{\sigma\in S_{p+q+r+1}}\sgn\sigma\sum_{\tau\in S_{p+q+1}^{i,p+q}}(-1)^{p+q+1-i}\sgn\tau\\
&\times\alpha(u_{\sigma(\tau(0))},\cdots,\widehat{u_{\sigma(\tau(i))}},\cdots,u_{\sigma(\tau(p))},u_{\sigma(\tau(p+1))})\\
&\times\beta(u_{\sigma(\tau(p))},\widehat{u_{\sigma(\tau(p+1))}},\cdots,u_{\sigma(\tau(p+q))},u_{\sigma(\tau(i))})\gamma(u_{\sigma(p+q)},\cdots,u_{\sigma(p+q+r)})\\
=&\sum_{i=0}^{p-1}\sum_{\sigma\in S_{p+q+r+1}}\sgn\sigma\sum_{\tau\in S_{p+q+1}^{i,p+q}}(-1)^{p+q-i}\sgn\tau\\
&\times\alpha(u_{\sigma(\tau(0))},\cdots,\widehat{u_{\sigma(\tau(i))}},\cdots,u_{\sigma(\tau(p+1))},u_{\sigma(\tau(p))})\\
&\times\beta(u_{\sigma(\tau(p))},\widehat{u_{\sigma(\tau(p+1))}},\cdots,u_{\sigma(\tau(p+q))},u_{\sigma(\tau(i))})\gamma(u_{\sigma(p+q)},\cdots,u_{\sigma(p+q+r)})\\
=&-\sum_{i=0}^{p-1}\sum_{\sigma\in S_{p+q+r+1}}\sgn\sigma\sum_{\mu\in S_{p+q+1}^{p+q}}\sgn\mu \alpha(u_{\sigma(\mu(0))},\cdots,u_{\sigma(\mu(p))})\beta(u_{\sigma(\mu(p))},\cdots,u_{\sigma(\mu(p+q))})\\
&\times\gamma(u_{\sigma(p+q)},\cdots,u_{\sigma(p+q+r)})\\
=&-J_1,
\end{split}
\end{equation*}
with the relation of $\mu$ and $\tau$ given by $\mu(0)=\tau(0),\cdots,\mu(i-1)=\tau(i-1),\mu(i)=\tau(i+1),\cdots,\mu(p-2)=\tau(p-1),\mu(p-1)=\tau(p+1),\mu(p)=\tau(p),\mu(p+1)=\tau(p+2),\cdots,\mu(p+q-1)=\tau(p+q), \mu(p+q)=\tau(i)=p+q$, and
\begin{equation*}
\begin{split}
J_3=&(-1)^{q-1}\sum_{i=0}^{p-1}\sum_{\sigma\in S_{p+q+r+1}}\sgn\sigma\sum_{\tau\in S_{p+q+1}^{i,p+q}}\sgn\tau \alpha(u_{\sigma(\tau(0))},\cdots,\widehat{u_{\sigma(\tau(p))}},u_{\sigma(\tau(p+1))})\\
&\times\beta(u_{\sigma(\tau(p))},\widehat{u_{\sigma(\tau(p+1))}},\cdots,u_{\sigma(\tau(p+q))},u_{\sigma(\tau(i))})\gamma(u_{\sigma(p+q)},\cdots,u_{\sigma(p+q+r)})\\
=&\sum_{i=0}^{p-1}\sum_{\sigma\in S_{p+q+r+1}}\sgn\sigma\sum_{\tau\in S_{p+q+1}^{i,p+q}}(-1)^{p+1-i}\sgn\tau\\
&\times\alpha(u_{\sigma(\tau(0))},\cdots,\widehat{u_{\sigma(\tau(i))}},\cdots,\widehat{u_{\sigma(\tau(p))}},u_{\sigma(\tau(p+1))},u_{\sigma(\tau(i))})\\
&\times\beta(u_{\sigma(\tau(i))},u_{\sigma(\tau(p))},\widehat{u_{\sigma(\tau(p+1))}},\cdots,u_{\sigma(\tau(p+q))})\gamma(u_{\sigma(p+q)},\cdots,u_{\sigma(p+q+r)})\\
=&p\sum_{\sigma\in S_{p+q+r+1}}\sgn\sigma\sum_{\mu\in S_{p+q+1}^{p,p+q}}\sgn\mu \alpha(u_{\sigma(\mu(0))},\cdots,u_{\sigma(\mu(p))})\beta(u_{\sigma(\mu(p))},\cdots,u_{\sigma(\mu(p+q))})\\
&\times\gamma(u_{\sigma(p+q)},\cdots,u_{\sigma(p+q+r)})\\
=&p I_2,
\end{split}
\end{equation*}
with the relation of $\mu$ and $\tau$ given by $\mu(0)=\tau(0),\cdots,\mu(i-1)=\tau(i-1),\mu(i)=\tau(i+1),\cdots,\mu(p-2)=\tau(p-1),\mu(p-1)=\tau(p+1),\mu(p)=\tau(i)=p+q,\mu(p+1)=\tau(p),\mu(p+2)=\tau(p+2),\cdots,\mu(p+q)=\tau(p+q)$.
The other terms of $I_1$ vanish because there are the same pair of vertices in the arguments of $\alpha$ and $\beta$ in those terms.
\begin{equation*}
\begin{split}
I_2=&\sum_{\sigma\in S_{p+q+r+1}}\sgn\sigma\sum_{\tau\in S_{p+q+1}^{p,p+q}}\sgn\tau\alpha(u_{\sigma(\tau(0))},\cdots,u_{\sigma(\tau(p))})\beta(u_{\sigma(\tau(p))},\cdots,u_{\sigma(\tau(p+q))})\\
&\times\gamma(u_{\sigma(p+q)},\cdots,u_{\sigma(p+q+r)})\\
=&\sum_{\sigma\in S_{p+q+r+1}}\sgn\sigma\sum_{\tau\in S_{p+q+1}^{p,p+q}}\sgn\tau\sum_{i=0}^{p}(-1)^{p+q-i}\alpha(u_{\sigma(\tau(0))},\cdots,\widehat{u_{\sigma(\tau(i))}},\cdots,u_{\sigma(\tau(p+1))})\\
&\times\beta(u_{\sigma(\tau(p+1))},\cdots,u_{\sigma(\tau(p+q))},u_{\sigma(
\tau(p))})\gamma(u_{\sigma(p+q)},\cdots,u_{\sigma(p+q+r)})\\
=&\sum_{\sigma\in S_{p+q+r+1}}\sgn\sigma\sum_{\tau\in S_{p+q+1}^{p,p+q}}(-1)^{q}\sgn\tau\alpha(u_{\sigma(\tau(0))},\cdots,\widehat{u_{\sigma(\tau(p))}},u_{\sigma(\tau(p+1))})\\
&\times\beta(u_{\sigma(\tau(p+1))},\cdots,u_{\sigma(\tau(p+q))},u_{\sigma(
\tau(p))})\gamma(u_{\sigma(p+q)},\cdots,u_{\sigma(p+q+r)})\\
=&\sum_{\sigma\in S_{p+q+r+1}}\sgn\sigma\sum_{\mu\in S_{p+q+1}^{p+q,p+q}}\sgn\mu\alpha(u_{\sigma(\mu(0))},\cdots,u_{\sigma(\mu(p))})\beta(u_{\sigma(\mu(p))},\cdots,u_{\sigma(\mu(p+q))})\\
&\times\gamma(u_{\sigma(p+q)},\cdots,u_{\sigma(p+q+r)})\\
=&(p+q)!\sum_{\sigma\in S_{p+q+r+1}}\sgn\sigma\alpha(u_{\sigma(0)},\cdots,u_{\sigma(p)})\beta(u_{\sigma(p)},\cdots,u_{\sigma(p+q)})\\
&\times\gamma(u_{\sigma(p+q)},\cdots,u_{\sigma(p+q+r)})\\
=&(p+q+r+1)!(p+q)!\bA(\alpha\otimes\beta\otimes\gamma)(u_0,u_1,\cdots,u_{p+q+r+1})
\end{split}
\end{equation*}
with the relation of $\tau$ and $\mu$ given by $\mu(0)=\tau(0),\cdots,\mu(p-1)=\tau(p-1),\mu(p)=\tau(p+1),\cdots,\mu(p+q-1)=\tau(p+q),\mu(p+q)=\tau(p)=p+q$.

Finally,
\begin{equation*}
\begin{split}
I_3=&\sum_{\sigma\in S_{p+q+r+1}}\sgn\sigma\sum_{i=p+1}^{p+q}\sum_{\tau\in S_{p+q+1}^{i,p+q}}\sgn\tau\alpha(u_{\sigma(\tau(0))},\cdots,u_{\sigma(\tau(p))})\\
&\times\beta(u_{\sigma(\tau(p))},\cdots,u_{\sigma(\tau(p+q))})\gamma(u_{\sigma(p+q)},\cdots,u_{\sigma(p+q+r)})\\
=&\sum_{\sigma\in S_{p+q+r+1}}\sgn\sigma\sum_{i=p+1}^{p+q}\sum_{\tau\in S_{p+q+1}^{i,p+q}}(-1)^{p+q-i}\sgn\tau\alpha(u_{\sigma(\tau(0))},\cdots,u_{\sigma(\tau(p))})\\
&\times\beta(u_{\sigma(\tau(p))},\cdots,\widehat{u_{\sigma(\tau(i))}},\cdots,u_{\sigma(\tau(p+q))}, u_{\sigma(\tau(i))})\gamma(u_{\sigma(p+q)},\cdots,u_{\sigma(p+q+r)})\\
=&\sum_{i=p+1}^{p+q}\sum_{\sigma\in S_{p+q+r+1}}\sgn\sigma\sum_{\mu\in S_{p+q+1}^{p+q,p+q}}\sgn\mu\alpha(u_{\sigma(\mu(0))},\cdots,u_{\sigma(\mu(p))})\\
&\times\beta(u_{\sigma(\mu(p))},\cdots,u_{\sigma(\mu(p+q))})\gamma(u_{\sigma(p+q)},\cdots,u_{\sigma(p+q+r)})\\
=&q(p+q)!(p+q+r+1)!\bA(\alpha\otimes\beta\otimes\gamma)(u_0,u_1,\cdots,u_{p+q+r}),
\end{split}
\end{equation*}
with the relation of $\mu$ and $\tau$ given by $\mu(0)=\tau(0),\cdots,\mu(i-1)=\tau(i-1),\mu(i)=\tau(i+1),\cdots,\mu(p+q-1)=\tau(p+q),\mu(p+q)=\tau(i)=p+q$.

Combining the expressions above together, we have
\begin{equation}
(\alpha\wedge\beta)\wedge\gamma=\bA(\alpha\otimes\beta\otimes\gamma).
\end{equation}
Similarly, we have
\begin{equation}
\alpha\wedge(\beta\wedge\gamma)=\bA(\alpha\otimes\beta\otimes\gamma).
\end{equation}
This completes the proof of (4).
\end{proof}
In the second part of this section, we present some details of computation in the proof of Theorem \ref{thm-gen}. We first need the following relation of parallel differential forms on $\R^n$ and forms for graphs embedded into $\R^n$.
\begin{prop}\label{prop-em}
Let $G$ be a graph embedded in $\R^n$. Let $\Phi:\bigwedge^k(\R^n)^*\to A^{k}(G)$ be defined as follows:
\begin{equation}
(\Phi\omega)(u_0,u_1,\cdots,u_k)=\omega(u_1-u_0,u_2-u_0,\cdots,u_k-u_0)
\end{equation}
when $\{u_0,u_1,\cdots,u_k\}\in C_{k+1}(G)$. Then,
\begin{equation}
\Phi\alpha\wedge \Phi\beta=\frac{k!l!}{(k+l)!}\Phi(\alpha\wedge\beta).
\end{equation}
for any $\alpha\in \bigwedge^k(\R^n)^*$ and $\beta\in \bigwedge^l(\R^n)^*$.
\end{prop}
\begin{proof}
For any $\{u_0,u_1,\cdots,u_k,u_{k+1},\cdots,u_{k+l}\}\in C_{k+l}(G)$,
\begin{equation*}
\begin{split}
&\Phi\alpha\wedge \Phi\beta(u_0,u_1,\cdots,u_{k+l})\\
=&\frac{1}{(k+l+1)!}\sum_{\sigma\in S_{k+l+1}}\sgn\sigma\cdot  \Phi\alpha(u_{\sigma(0)},u_{\sigma(1)},\cdots,u_{\sigma(k)})\Phi\beta(u_{\sigma(k)},\cdots,u_{\sigma(k+l)})\\
=&\frac{1}{(k+l+1)!}\sum_{\sigma\in S_{k+l+1}}\sgn\sigma\cdot\alpha(u_{\sigma(1)}-u_{\sigma(0)},\cdots,u_{\sigma(k)}-u_{\sigma(0)})\\
&\times\beta(u_{\sigma(k+1)}-u_{\sigma(k)},\cdots,u_{\sigma(k+l)}-u_{\sigma(k)})\\
=&\frac{1}{(k+l+1)!}\sum_{\sigma\in S_{k+l+1}}\sgn\sigma\cdot\alpha(u_{\sigma(1)}-u_0+u_0-u_{\sigma(0)},\cdots,u_{\sigma(k)}-u_0+u_0-u_{\sigma(0)})\\
&\times\beta(u_{\sigma(k+1)}-u_0+u_0-u_{\sigma(k)},\cdots,u_{\sigma(k+l)}-u_0+u_0-u_{\sigma(k)})\\
=&\frac{1}{(k+l+1)!}(I_1+I_2+I_3+I_4+I_5)\\
\end{split}
\end{equation*}
where
\begin{equation*}
\begin{split}
I_1=&\sum_{\sigma\in S_{k+l+1}}\sgn\sigma\cdot\alpha(u_{\sigma(1)}-u_0,\cdots,u_{\sigma(k)}-u_0)\beta(u_{\sigma(k+1)}-u_0,\cdots,u_{\sigma(k+l)}-u_0)\\
=&\sum_{\sigma\in S^{0}_{k+l+1}}\sgn\sigma\cdot\alpha(u_{\sigma(1)}-u_0,\cdots,u_{\sigma(k)}-u_0)\beta(u_{\sigma(k+1)}-u_0,\cdots,u_{\sigma(k+l)}-u_0),\\
=&k!l!\Phi(\alpha\wedge\beta)(u_0,u_1,\cdots,u_{k+l}),
\end{split}
\end{equation*}
\begin{equation*}
\begin{split}
I_2=&\sum_{i=1}^k\sum_{\sigma\in S_{k+l+1}}\sgn\sigma\cdot\alpha(u_{\sigma(1)}-u_0,\cdots,\widehat{u_{\sigma(i)}-u_0},u_0-u_{\sigma(0)},\cdots,u_{\sigma(k)}-u_0)\\
&\times\beta(u_{\sigma(k+1)}-u_0,\cdots,u_{\sigma(k+l)}-u_0)\\
=kI_1
\end{split}
\end{equation*}
by swapping the values of $\sigma(0)$ and $\sigma(i)$,
\begin{equation*}
\begin{split}
I_3=&\sum_{i=k+1}^{k+l}\sum_{\sigma\in S_{k+l+1}}\sgn\sigma\cdot\alpha(u_{\sigma(1)}-u_0,\cdots,u_{\sigma(k)}-u_0)\\
&\times\beta(u_{\sigma(k+1)}-u_0,\cdots,\widehat{u_{\sigma(i)}-u_0},u_0-u_{\sigma(k)},\cdots,u_{\sigma(k+l)}-u_0)\\
=&0
\end{split}
\end{equation*}
by swapping the values of $\sigma(i)$ and $\sigma(0)$,
\begin{equation*}
\begin{split}
I_4=&\sum_{i=1}^{k-1}\sum_{j=k+1}^{k+l}\sum_{\sigma\in S_{k+l+1}}\sgn\sigma\cdot\alpha(u_{\sigma(1)}-u_0,\cdots,\widehat{u_{\sigma(i)}-u_0},u_0-u_{\sigma(0)},\cdots,u_{\sigma(k)}-u_0)\\
&\times\beta(u_{\sigma(k+1)}-u_0,\cdots,\widehat{u_{\sigma(j)}-u_0},u_0-u_{\sigma(k)},\cdots,u_{\sigma(k+l)}-u_0)\\
=0
\end{split}
\end{equation*}
by swapping the value of $\sigma(i)$ and $\sigma(j)$, and
\begin{equation*}
\begin{split}
I_5=&\sum_{j=k+1}^{k+l}\sum_{\sigma\in S_{k+l+1}}\sgn\sigma\cdot\alpha(u_{\sigma(1)}-u_0,\cdots,u_{\sigma(k-1)}-u_0,u_0-u_{\sigma(0)})\\
&\times\beta(u_{\sigma(k+1)}-u_0,\cdots,\widehat{u_{\sigma(j)}-u_0},u_0-u_{\sigma(k)},\cdots,u_{\sigma(k+l)}-u_0)\\
=&\sum_{j=k+1}^{k+l}\sum_{\sigma\in S_{k+l+1}}\sgn\sigma\cdot\alpha(u_{\tau(1)}-u_0,\cdots,u_{\tau(k-1)}-u_0,u_{\sigma(k)}-u_0)\\
&\times\beta(u_{\tau(k+1)}-u_0,\cdots,u_{\tau(j)}-u_0,\cdots,u_{\tau(k+l)}-u_0)\\\\
=&l I_1
\end{split}
\end{equation*}
by re-ordering $\sigma$ to $\tau$ with $\tau(0)=\sigma(j)$, $\tau(k)=\sigma(0)$ ,$\tau(j)=\sigma(k)$ and $\tau(i)=\sigma(i)$ for $i\not\in\{0, k,j\}$.

Therefore,
\begin{equation}
\Phi\alpha\wedge\Phi\beta=\frac{k!l!}{(k+l)!}\Phi(\alpha\wedge\beta).
\end{equation}
\end{proof}

With the help of Proposition \ref{prop-em}, we give the details of computation in the proof of Theorem \ref{thm-gen}.
\begin{prop}\label{prop-compute}
Let $G$ be the graph of standard tessellation of $\R^n$, $1\leq k\leq n-1$.
 Then,
\begin{enumerate}
\item For each $1\leq i_1<i_2<\cdots<i_{k+1}\leq n$, $dx_{i_1}\wedge dx_{i_2}\wedge\cdots\wedge dx_{i_{k+1}}(u_0,u_1,\cdots,u_{k+1})=\pm\frac{1}{(k+1)!}$ or $0$.
\item For each $1\leq i_1<i_2<\cdots<i_{k+1}\leq n$ and each $(k+1)$-clique $\{u_0,\cdots,u_{k}\}$ in $G$,
\begin{equation}
\begin{split}
&\left|\left\{u\in G\ :\ dx_{i_1}\wedge\cdots dx_{i_{k+1}}(u_0,u_1,\cdots,u_{k},u)=\frac{1}{(k+1)!}\right\}\right|\\
=&\left|\left\{u\in G\ :\ dx_{i_1}\wedge\cdots dx_{i_{k+1}}(u_0,u_1,\cdots,u_{k},u)=-\frac{1}{(k+1)!}\right\}\right|\\
\leq& 2^{n-k-1}.
\end{split}
\end{equation}
\end{enumerate}
\end{prop}
\begin{proof}
(1) Equip with $\mathbb Z^n$ the partial order: $x\leq y$ if $y_i-x_i\geq 0$ for all $i=1,2,\cdots,n$. Then, a $(k+2)$-clique $\{u_0,u_1,\cdots,u_{k+1}\}$ of $G$ is a totally ordered subset of $\mathbb Z^n$. Without loss of generality, suppose that $u_0<u_1<\cdots<u_{k+1}$. Then, there is a flag $\emptyset\subsetneqq I_1\subsetneqq  I_2\cdots\subsetneqq I_{k+1}$ of $\{1,2,\cdots,n\}$, such that
\begin{equation}
u_i-u_0=e_{I_i}
\end{equation}
for $i=1,2,\cdots,k+1$. Here, for a subset $I$ of $\{1,2,\cdots,n\}$, $e_I=\sum_{i\in I}e_i$. Then, by Proposition \ref{prop-em},
\begin{equation}
\begin{split}
&dx_{i_1}\wedge\cdots \wedge dx_{i_{k+1}}(u_0,u_1,\cdots,u_{k+1})\\
=&\frac{1}{(k+1)!}dx_{i_1}\wedge\cdots \wedge dx_{i_{k+1}}(e_{I_1},e_{I_2},\cdots,e_{I_{k+1}})\\
=&\frac{1}{(k+1)!}dx_{i_1}\wedge\cdots\wedge dx_{i_{k+1}}(e_{J_1},e_{J_2},\cdots,e_{J_{k+1}})\\
=&\left\{\begin{array}{ll}\frac{\sgn\sigma}{(k+1)!}&\{i_{\sigma(j)}\}=J_j\ \mbox{for $j=1,2,\cdots,k+1$}\\
0&\mbox{otherwise}
\end{array}\right.
\end{split}
\end{equation}
where $J_i=(I_i\setminus I_{i-1})\cap\{i_1,i_2,\cdots,i_{k+1}\}$ for $i=1,2,\cdots,k+1$ with $I_{0}=\emptyset$.\\
(2) By the same argument as in (1), we may assume that $u_0=o,u_1=e_{I_1},\cdots,u_{k}=e_{I_k}$ where $\emptyset\subsetneqq I_1\subsetneqq  I_2\cdots\subsetneqq I_{k+1}$ is a flag of $\{1,2,\cdots,n\}$.
Then, by Proposition \ref{prop-em},
\begin{equation}
\begin{split}
&dx_{i_1}\wedge\cdots \wedge dx_{i_{k+1}}(u_0,u_1,\cdots,u_{k},u)\\
=&\frac{1}{(k+1)!}dx_{i_1}\wedge\cdots \wedge dx_{i_{k+1}}(e_{I_1},e_{I_2},\cdots,e_{I_{k}},u)\\
=&\frac{1}{(k+1)!}dx_{i_1}\wedge\cdots\wedge dx_{i_{k+1}}(e_{J_1},e_{J_2},\cdots,e_{J_{k}},u)\\
\end{split}
\end{equation}
where $J_i=(I_i\setminus I_{i-1})\cap\{i_1,i_2,\cdots,i_{k+1}\}$ for $i=1,2,\cdots,k$ with $I_{0}=\emptyset$.

If there is some $J_i=\emptyset$, then $$dx_{i_1}\wedge\cdots\wedge dx_{i_{k+1}}(e_{J_1},e_{J_2},\cdots,e_{J_{k}},u)=0$$ for all $u$. The conclusion is clearly true.

If $J_i$ is not empty for any $i=1,2,\cdots, k$, only the following two cases happen.
\begin{enumerate}
\item[(i)] Each $J_i$ contains only one element in $\{i_1,i_2,\cdots,i_{k+1}\}$. In this case, without loss of generality, we can assume that $J_j=i_j$ for $j=1,2,\cdots,k$. Then, to make sure that $$dx_{i_1}\wedge\cdots dx_{i_{k+1}}(u_0,u_1,\cdots,u_{k},u)\neq 0,$$ we must have $u<o$ or $u>e_{I_k}$. So,
\begin{equation*}
\begin{split}
&\left\{u\in G\ :\ dx_{i_1}\wedge\cdots dx_{i_{k+1}}(u_0,u_1,\cdots,u_{k},u)=\frac{1}{(k+1)!}\right\}\\
=&\{e_{I_k}+e_{i_{k+1}}+e_J\ :\ J\subset I_{k}^{c}\setminus\{i_{k+1}\}\}.
\end{split}
\end{equation*}
and
\begin{equation*}
\begin{split}
&\left\{u\in G\ :\ dx_{i_1}\wedge\cdots dx_{i_{k+1}}(u_0,u_1,\cdots,u_{k},u)=-\frac{1}{(k+1)!}\right\}\\
=&\{-e_{i_{k+1}}-e_J\ :\ J\subset I_{k}^{c}\setminus\{i_{k+1}\}\}.
\end{split}
\end{equation*}
Hence
\begin{equation*}
\begin{split}
&\left|\left\{u\in G\ :\ dx_{i_1}\wedge\cdots dx_{i_{k+1}}(u_0,u_1,\cdots,u_{k},u)=\frac{1}{(k+1)!}\right\}\right|\\
=&\left|\left\{u\in G\ :\ dx_{i_1}\wedge\cdots dx_{i_{k+1}}(u_0,u_1,\cdots,u_{k},u)=-\frac{1}{(k+1)!}\right\}\right|\\
=&2^{n-|I_k|-1}\\
\leq& 2^{n-k-1}.
\end{split}
\end{equation*}
\item[(ii)] Some $J_i$ contains two elements of $\{i_1,i_2,\cdots,i_{k+1}\}$. Without loss of generality, we may assume that $J_1=\{i_1\},\cdots, J_{k-1}=\{i_{k-1}\}$ and $J_k=\{i_{k},i_{k+1}\}$. To make sure that
    $$dx_{i_1}\wedge\cdots dx_{i_{k+1}}(u_0,u_1,\cdots,u_{k},u)\neq 0,$$ we must have $e_{I_{k-1}}<u<e_{I_k}$. So
\begin{equation*}
\begin{split}
&\left\{u\in G\ :\ dx_{i_1}\wedge\cdots dx_{i_{k+1}}(u_0,u_1,\cdots,u_{k},u)=\frac{1}{(k+1)!}\right\}\\
=&\{e_{I_{k-1}}+e_{i_{k+1}}+e_J\ :\ J\subset I_k\setminus(I_{k-1}\cup\{i_k,i_{k+1}\})\}.
\end{split}
\end{equation*}
and
\begin{equation*}
\begin{split}
&\left\{u\in G\ :\ dx_{i_1}\wedge\cdots dx_{i_{k+1}}(u_0,u_1,\cdots,u_{k},u)=-\frac{1}{(k+1)!}\right\}\\
=&\{e_{I_{k-1}}+e_{i_{k}}+e_J\ :\ J\subset I_k\setminus(I_{k-1}\cup\{i_k,i_{k+1}\})\}.
\end{split}
\end{equation*}
Hence
\begin{equation*}
\begin{split}
&\left|\left\{u\in G\ :\ dx_{i_1}\wedge\cdots dx_{i_{k+1}}(u_0,u_1,\cdots,u_{k},u)=\frac{1}{(k+1)!}\right\}\right|\\
=&\left|\left\{u\in G\ :\ dx_{i_1}\wedge\cdots dx_{i_{k+1}}(u_0,u_1,\cdots,u_{k},u)=-\frac{1}{(k+1)!}\right\}\right|\\
=&2^{|I_k|-|I_{k-1}|-2}\\
\leq& 2^{n-k-1}
\end{split}
\end{equation*}
since $|I_k|\leq n$ and $|I_{k-1}|\geq k-1$.
\end{enumerate}
This completes the proof of the proposition.
\end{proof}

\end{document}